\documentclass[11pt]{amsart}
\usepackage{amsmath,amsthm,amssymb,latexsym,graphicx,verbatim,enumerate,
ifpdf,mdwlist,xspace, bm}
\usepackage{mathabx}
\ifpdf
\usepackage{hyperref}%[unicode]{hyperref}
\else
\usepackage[hypertex]{hyperref}
\fi
\usepackage{color}
\usepackage{amssymb, amsmath, amsthm}
\usepackage{graphicx}
\usepackage{cancel}
\usepackage{float}

\usepackage{verbatim}

\newcommand{\er}{equivalence relation\xspace}
\newcommand{\ers}{equivalence relations\xspace}
\newcommand{\NCA}{noncomputably avoiding\xspace}

\providecommand{\notleftright}{\mathrel{\ooalign{$\Leftrightarrow$\cr\hidewidth$/$\hidewidth}}}

\providecommand{\biota}{\bm{\iota}}

\providecommand{\I}{\mathcal{I}}
\providecommand{\bI}{\bm{\I}}
\providecommand{\F}{\mathcal{F}}
\providecommand{\bF}{\bm{\F}}

\providecommand{\leqc}{\leq_c}
\providecommand{\nleqc}{\nleq_c}
\providecommand{\equivc}{\equiv_c}
\providecommand{\nequivc}{\not\equiv_c}

\newcommand{\evens}{\text{Evens}}
\newcommand{\odds}{\text{Odds}}

\providecommand{\set}[1]{\lbrace #1 \rbrace}

\providecommand{\abs}[1]{\lvert#1\rvert}

\providecommand{\dom}{{\rm{dom}}}

\DeclareMathOperator{\Th}{Th}

\DeclareMathOperator{\ran}{\range}
\DeclareMathOperator{\orb}{orb}
\DeclareMathOperator{\Id}{Id}
\DeclareMathOperator{\bId}{\mathbf{Id}}
\DeclareMathOperator{\ER}{\mathbf{ER}}

\DeclareMathOperator{\col}{Col}

\newtheorem{question}{Question}

\newtheorem{convention}{Convention}
\newtheorem{thm}{Theorem}[section]
\newtheorem{definition}[thm]{Definition}
\newenvironment{rem}{\begin{remark} \rm}{ \end{remark}}
\newtheorem{remark}[thm]{Remark}
\newtheorem{lemma}[thm]{Lemma}

\newtheorem{obs}[thm]{Observation}

\newtheorem{corollary}[thm]{Corollary}
\newtheorem{cory}[thm]{Corollary}

\newtheorem{claim}[thm]{Claim}

\newenvironment{defn}{\begin{definition} \rm}{ \end{definition}}
\newtheorem{lem}[thm]{Lemma}

\renewcommand{\phi}{\varphi}
\renewcommand{\setminus}{\smallsetminus}

\DeclareMathOperator{\Dark}{Dark}
\DeclareMathOperator{\Light}{Light}

\DeclareMathOperator{\bDark}{\mathbf{Dark}}
\DeclareMathOperator{\bLight}{\mathbf{Light}}

\DeclareMathOperator{\range}{range}
\DeclareMathOperator{\graph}{graph}
\DeclareMathOperator{\Ceers}{\mathbf{Ceers}}

\newcommand{\rel}[1]{\mathrel{#1}}

\usepackage{color, soul}
\usepackage[normalem]{ulem}

\title[Computable reducibility on equivalence relations on $\mathbb{N}$]{On the structure of computable reducibility on equivalence relations of natural numbers}

\author[U.~Andrews]{Uri Andrews}
\address{Department of Mathematics\\
University of Wisconsin\\
Madison, WI 53706-1388\\
USA}
\email{\href{mailto:andrews@math.wisc.edu}{andrews@math.wisc.edu}}

\author[D.~Belin]{Daniel F. Belin}
\address{Department of Mathematics\\
University of Wisconsin\\
Madison, WI 53706-1388\\
USA}
\email{\href{mailto:dbelin@wisc.edu}{dbelin@wisc.edu}}

\author[L.~San Mauro]{Luca San Mauro}
\address{Department of  Mathematics ``Guido Castelnuovo''\\
Sapienza University of Rome}
\email{\href{mailto:luca.sanmauro@uniroma1.it}{luca.sanmauro@uniroma1.it}}
\thanks{Andrews was partially supported by NSF grant DMS-1600228. San Mauro was partially supported by the Austrian Science Fund, project M~2461.}



\begin{document}

\maketitle

\begin{abstract}
	We examine the degree structure $\ER$ of equivalence relations on $\omega$ under computable reducibility. We examine when pairs of degrees have a {least upper bound}. In particular, we show that sufficiently incomparable pairs of degrees do not have a least upper bound but that some incomparable degrees do, and we characterize the degrees which have a least upper bound with every finite equivalence relation. We show that the natural classes of finite, light, and dark degrees are definable in $\ER$. We show that every equivalence relation has continuum many self-full strong minimal covers, and that $\mathbf{d}\oplus \mathbf{\bId_1}$ needn't be a strong minimal cover of a self-full degree $\mathbf{d}$.
	Finally, we show that the theory of the degree structure $\ER$ as well as the theories of the substructures of light degrees and of dark degrees are each computably isomorphic with second order arithmetic.
\end{abstract}

\section{Introduction}
The study of the complexity of equivalence relations has been a major thread of research in diverse areas of logic. The most popular way  for evaluating this complexity is by defining a suitable reducibility. A reduction of an equivalence relation
$R$ on a domain $X$ to an equivalence relation $S$ on a domain $Y$
is a (nice) function $f: X\rightarrow Y$
such that 
\[
x\rel{R} y \Leftrightarrow f(x)\rel{S}f(y).
\]
That is, $f$ pushes down to an injective map on the quotient sets  $X_{R}\mapsto Y_{S}$. It is natural to impose a bound on the complexity of the reduction $f$, as  otherwise, if the size of {$X_R$} is not larger than the size of $X_S$, then the Axiom of Choice alone would guarantee the existence of a reduction from $R$ to $S$, 
thus we would not be able to
distinguish equivalence relations with the same number of equivalence classes.
In the literature, there are two main definitions for this reducibility, designed to deal, respectively, with the uncountable and the countable case:

\begin{itemize}
\item In descriptive set theory,  \emph{Borel reducibility} $(\leq_B)$ is defined by assuming that $X$ and $Y$ are Polish spaces and $f$ is Borel;
\item In computability theory, \emph{computable reducibility} $(\leq_c)$ is defined by assuming that $X=Y$ coincide with the set $\omega$ of natural numbers and $f$ is computable. 
\end{itemize}

The theory of Borel equivalence relations (as surveyed in, e.g, \cite{gao2008invariant,hjorth2010borel})  is a central field of modern descriptive set theory and it  shows deep connections with topology, group theory, combinatorics, model theory and ergodic theory---to name a few.
% see, e.g., \cite{friedman1989borel} where the notion has been introduced;  \cite{gao2008invariant,hjorth2010borel} where the notion is surveyed; the papers \cite{LR-05, FM-11} for particular examples of equivalence relations which are complete; \cite{HKL, CK} for results about the structure of Borel reducibility. The theory of Borel equivalence relations  is now a central and rapidly expanding field of modern descriptive set theory, that shows deep connections with topology, group theory, combinatorics, model theory and ergodic theory---to name a few.

Research on computable reducibility dates back to the work of Ershov~\cite{Ershov:77,Ershov:survey}  {and} the theory of numberings. It concentrates on two main focuses: first, to calculate the complexity of natural equivalence relations on $\omega$, proving, e.g., that provable equivalence in Peano Arithmetic is $\Sigma^0_1$-complete~\cite{Bernardi:83},  Turing equivalence on c.e.\ sets is $\Sigma^0_4$-complete~\cite{IMNN-14}, and the isomorphism relations on several familiar classes of computable structures (e.g., trees, torsion abelian groups, fields of characteristic $0$ or $p$) are $\Sigma^1_1$-complete~\cite{FFH-12}; secondly, 
to understand the structure of the collection of equivalence relations of a certain complexity $\Gamma$ (e.g., lying at some level of the arithmetical~\cite{coskey2012hierarchy}, analytical~\cite{bazhenov2020minimal}, or Ershov hierarchy~\cite{bazhenov2020classifying,ng2019degree}).

Regarding the latter focus, computably enumerable equivalence relations---known by the acronym \emph{ceers}~\cite{gao2001computably}, or called \emph{positive} equivalence relations in the Russian literature---received special attention. Historically, the emphasis was on combinatorial classes of \emph{universal ceers}, i.e., ceers to which all other ceers computably reduce (see, e.g., \cite{andrews2014universal,andrews2017survey}). But recently, there has been a growing interest in pursuing a systematic study of \textbf{Ceers}, the poset of degrees of ceers, whose structure turns out to be extremely rich. Andrews, Schweber, and Sorbi~\cite{andrews2020theory} proved that the first-order theory of \textbf{Ceers} is as complicated as true arithmetic (see also \cite{andrews2019joins} for a structural analysis of  \textbf{Ceers} focused on joins, meets, and definability).

In this paper, we focus rather on $\ER$, the poset of degrees of \emph{all} equivalence relations with domain $\omega$. Our interest in $\ER$ is twofold. 

On the one hand, we want to explore to what extent techniques coming from  the theory of ceers can be applied to equivalence relations of arbitrary complexity. Some proofs will move smoothly from $\mathbf{Ceers}$ to $\ER$ (proving that the underlying results are independent from the way in which the equivalence relations are presented), but the analogy between the two structures often breaks down (see, e.g., Theorem \ref{thm:plusId1NotAStrongCover}), or   new ideas will be required to recast analogous results from the setting of ceers (see, e.g., Theorem \ref{definability of Id}).

On the other hand, we regard $\ER$ as a natural structure, interesting and worth studying \emph{per se}. After all, $\ER$ is to $\Ceers$ as, e.g., the global structure of all Turing degrees ($\mathcal{D}_T$) is to the local structure of c.e.\ degrees ($\mathcal{R}_T$)---and we consider it only a historical anomaly that, for equivalence relations, the local structure has been analyzed in great detail with no parallel investigation of the global structure.

We add a final piece of motivation. Dealing with a seemingly distant problem (i.e., Martin's conjecture), Bard \cite{bard2020uniform} recently proved that $\mathcal{D}_T$ is Borel reducible to  $\ER$. This may be regarded as  evidence that $\ER$ is complex. In this paper, we push this analysis further by fully characterizing the complexity  of the theory of $\ER$ (Theorem \ref{ER equivalent to 2order arithmetic}).

The rest of this paper is organized as follows. In the remainder of this section, we offer a number of preliminaries to make the paper  self-contained. In Section~\ref{sec:definability}, we focus on first-order definability of some natural fragments of $\ER$ and analyze when least upper bounds exist. In Section~\ref{sec:covers}, we study minimal and strongly minimal covers of equivalence relations, and we also introduce generic covers. Through this study,  we exhibit many disanalogies between $\ER$ and $\Ceers$. 
Finally, in Section~\ref{sec:arithmetic}, we show that the first-order theory of $\mathbf{ER}$ (and in fact, that of two natural fragments of $\mathbf{ER}$) is as {complex} as possible, being computably isomorphic to second-order arithmetic.

% the study of word problems for groups~\cite{nies2016,miller1971group}; the theory of numberings~\cite{Ershov:77,Ershov:survey}; the metamathematics of the arithmetics~\cite{Visser:80,Montagna:82,Bernardi:83}; computable structure theory~\cite{Fokina:12b}---and so forth. 

%The goal of this paper is to study the poset $(\mathbf{ER}, \leq_c)$ of the $c$-degrees of the equivalence relations with domain $\omega$.

%As usual, degrees will be denoted in boldface.

%\texttt{I will  expanded it later by introducing local structures and the global structure and cite more papers~\cite{andrews2020theory,bard2019uniform}}

\smallskip

Our computability theoretic terminology and notation is standard, and as in \cite{soare2016turing}.

%\luca{As suggested by the first reviewer, I made the following global changes:}
%\begin{enumerate}
%\item \lucainsert{I  replaced all the ``:'' occurring in the definitions of functions or reductions with the LaTeX command} \verb=\colon=
%\item \lucainsert{I replaced $\mu_f$ (and similar) with $f^\star$.}
%\item \lucainsert{I changed most instances of ``computably reduces to'' with ``is computably reducible to''}.
%\end{enumerate}

\subsection{Preliminary material} 
Throughout this subsection we assume that $R$ and $S$ are equivalence relations.  The $R$-equivalence class of a natural number $x$ is denoted by
$[x]_R$. For a set $A\subseteq \omega$, the \emph{$R$-saturation} of $A$ (i.e., $\bigcup_{x\in A}[x]_R$) is denoted by $[A]_R$. We denote the collection of all $R$-equivalence classes
 by $\omega_R$. 
 If $f$ is a computable function witnessing that $R\leqc S$, then we write $f\colon R\leqc S$. If $f\colon R \leqc S$, then  $f^\star$ is the injective mapping from $\omega_R$ to $\omega_S$ induced by $f$. In our proofs, it will sometimes be useful to consider the \emph{orbit} of a number or of an equivalence class along all  iterations of a given reduction: for $x\in\omega$ and $X\in\omega_R$, denote by $\orb_f(x)$ the set $\{f^{(i)}(x) : i>0 \}\subseteq \omega$ and by $\orb_f(X)$ the set $\{{f^\star}^{(i)}(X) : i>0 \}\subseteq \omega_R$.   The following lemma, which is immediate to prove, will be used many times in the paper, often implicitly.

\begin{lemma}\label{f respects m-degrees}
Let $f \colon  R \leqc S$. For all $X\in\omega_R$, $X \leq_m f^\star(X)$ so also $X\leq_m S$.
\end{lemma}

\begin{defn}
	For any nonempty c.e.\ set $W$ and \er $R$, we let $R\restriction W$ be the \er given by $x \rel{R\restriction W} y$ if and only if $h(x)\rel{R} h(y)$, where $h\colon \omega\rightarrow W$ is any computable surjection (note that up to $\equivc$, the definition does not depend on the choice of surjection $h$). 
\end{defn}

\begin{rem}\label{rmk:restrictions}
	For any nonempty c.e.\ set $W$ and \er $R$, observe that $h$ (as in the definition) gives a reduction of $R\restriction W$ to $R$, which we call the inclusion map. 
	Also, if $f\colon  X\leqc Y$, then $X \equivc Y\restriction \ran(f)$.
\end{rem}

If $f\colon  R\leqc S$ and $\range(f) \cap X \neq \emptyset$ for some $X\in \omega_S$, then we say that $f$ \emph{hits} $X$; otherwise, we say that $f$ \emph{avoids} $X$. 
We say that $R$ is \emph{self-full} if every reduction of  $R$ to itself hits all elements of $\omega_R$. The notion of self-fullness plays a prominent role in the theory of ceers (see, e.g.,\cite{andrews2019joins,andrews2020self,andrews2020theory}). To name just a couple of examples: the degrees of self-full ceers are definable in \textbf{Ceers}, as they coincide  with the nonuniversal  degrees which are meet-irreducible; moreover, the existence of self-full strong minimal covers is fundamental to prove  that the first-order theory of the degrees of light ceers is computably isomorphic to true arithmetic.

\smallskip

 By the notation $f\oplus g$, we denote the  following function,
 \[
 f\oplus g(x)=\begin{cases}
f(x) &\text{if $x$ is even},\\
g(x) &\text{if $x$ is odd}. 
 \end{cases}
 \]

The \emph{uniform join}\footnote{To avoid potential ambiguities between the terms ``uniform join'' and ``join'', we use the term ``least upper bound'' to refer to a join of degrees in the poset \textbf{ER}.} $R\oplus S$ is the equivalence relation that encodes $R$ on the evens and $S$ on the odds, i.e.,
$x\rel{ R\oplus S} y$ if and only if either $x=2u, y=2v$, and  $u\rel{R}v$; or $x=2u+1$, $y=2v+1$, and $u\rel{S}v$.
For the sake of exposition, we often say \emph{$R$-classes} (respectively, \emph{$S$-classes}) for the equivalence classes of $R\oplus S$ consisting of even (odd) numbers. The operation $\oplus$ is clearly associative, up to $\equiv_c$, 
%\lucaout{on degrees}\uri{It's not associative on equivalence relations -- in $(X\oplus Y) \oplus Z$, $Z$ is coded on the whole odds, whereas in $X\oplus (Y\oplus Z)$, $Z$ is encoded on the numbers congruent to $3$ mod $4$.}, 
so we will generally be lax and write expressions such as $R_0\oplus\ldots \oplus R_n$.

The following easy lemma 
was stated for ceers in \cite[Fact~2.3]{andrews2020theory} but goes through for arbitrary equivalence
relations with exactly the same proof.

\begin{lemma}\label{lem:below a join}
	If $X\leqc R\oplus S$, then there are $R_0\leqc R$ and $S_0\leqc S$ such that $X\equivc R_0 \oplus S_0$.
\end{lemma}

\begin{proof}
	Let $f\colon  X\leqc R\oplus S$ and denote $\range(f)$ by $W$. Then 
	\[
	X\equivc R\oplus S \restriction W \equivc R\restriction V_1 \oplus S \restriction V_2,
	\]
	where $V_1:=\{x:2x\in W\}$ and $V_2:=\{x: 2x+1\in W\}$.
\end{proof}

%It may also be convenient to uniformly join an infinite collection of \ers. The standard way to do so is by using all columns of $\omega$. But in what follows we will need to choose the coding columns accordingly to some external information we wish to encode. So, here is a general definition:
%\begin{defn}
%Given a countable collection of \ers $(R_i)_{i\in\omega}$ and a set of numbers $A:=\{a_0,a_1,\ldots\}$, we define $\oplus_i R_i[A]$ to be the ceer so that
%\[
%\langle j,x\rangle\rel{\oplus_i R_i[A]}\langle k, y\rangle \Leftrightarrow (\exists i)(j=k=a_i \mbox{ $\&$ } xR_iy).
%\]
%\end{defn}

If $A\subseteq \omega\times\omega$, then $R_{/A}$ is the equivalence relation generated by the set of pairs $R\cup A$. We say that $R_{/A}$ is a \emph{quotient} of $R$, and a quotient is \emph{proper} if $R_{/A}\neq R$. To improve readability, we often omit braces, e.g., writing $R_{(x,y)}$ instead of $R_{\{(x,y)\}}$.
Of particular interest for this paper will be  
 quotients of uniform joins.
  A quotient  ${R\oplus S}_{/A}$ is \emph{pure}  if it does not  collapse distinct $R$-classes, or distinct $S$-classes, i.e.,
\[
{R\oplus S}_{/A} \restriction \rel{\evens } \, = {R\oplus S} \restriction \rel{\evens} \mbox{ and } {R\oplus S}_{/A} \restriction \rel{\odds}\, = {R\oplus S} \restriction \rel{\odds}.
\]
The quotient $R\oplus S_{/A}$ is a \emph{total} quotient if every odd number is equivalent to an even number and vice versa.

\begin{lemma}\label{obs:pure quotients are ubs}
Every pure quotient of $R\oplus S$ is an upper bound of $R$ and $S$. 
\end{lemma}

\begin{proof}
Assume that $R\oplus S_{/A}$ is pure. It is immediate to observe that $R$ {is computably reducible}  to $R\oplus S/_{A}$ via the function $x \mapsto 2x$ and $S$ {is computably reducible} to $R\oplus S/_{A}$ via the function $x \mapsto 2x+1$.
\end{proof}

\begin{lem}\label{lem:reducing to total pure quotients and finitely in odds}
	Let $R\oplus S_{/A}$ be a total quotient of $R\oplus S$. Suppose that $f\colon  X\leqc R\oplus S_{/A}$ and $\ran(f)\cap \odds$ is finite.
%	, or that $S$ is a ceer and $\ran(f)\cap \odds$ hits only finitely many $S$-classes. 
Then $X\leqc R$.
\end{lem}
\begin{proof}
	For each $x\in \ran(f)\cap \odds$, fix an even number $x'$ so that $x\rel{R\oplus S_{/A}} x'$.
%	Similarly, if $S$ is a ceer and $\ran(f)$ intersects only finitely many $S$-classes, we can fix $(a_i)_{i\leq k}$ and $(b_i)_{i\leq k}$ so that the $a_i$'s represent each $S$-class in the range of $f$ and $b_i$ is even and equivalent to $a_i$. Then for any odd number $x\in \cup_{i\leq k}[a_i]_S$, let $x'=b_i$.
%	
	Let
	\[
	h(x) = 
	\begin{cases} f(x) &\text{if {$f(x)$} is even}\\
		f(x)' &\text{if {$f(x)$} is odd}\\
	\end{cases}
	\]
	and observe that $h$ is a reduction of $X$ to $R\oplus S_{/A}$ with range contained in the evens, so {$x\mapsto \frac{h(x)}{2}$} 
	%\lucaout{$\frac{h}{2}$} 
	is a reduction of $X$ to $R$.
\end{proof}

Let us now fix notation for some natural families of equivalence relations of natural numbers. They will serve as benchmark relations for our structural analysis of $\ER$. Some terminology naturally generalizes from the theory of ceers (see, e.g., \cite{andrews2019joins}).

\begin{itemize}
\item
Define $\Id_n$ by $x\mathrel{\Id_n} y$ if $x\equiv y \mod n$. 
Define $\Id=\Id_\omega$ by $x\Id y$ if $x=y$. For convenience in inductive arguments, we also consider $\Id_0$ to be the empty relation. We define $\mathcal{I}$ to be the family of equivalence relations that are equivalent to some $\Id_n$ for $1\leq n\in \omega$.
\item An equivalence relation $R$ is \emph{finite}, if $R$ has finitely many equivalence classes\footnote{This terminology, which is standard in the theory of ceers, differs from usage in descriptive set theory, where finite equivalence relations are those with
all equivalence classes being finite. In \cite{gao2001computably}, ceers with all  equivalence classes being finite are called $FC$ (standing for \emph{finite classes}).}. Otherwise $R$ is \emph{infinite}. $\mathcal{F}$ and $\mathcal{F}_n$ denote respectively the family of all finite equivalence relations and the family of equivalence relations with exactly $n$ equivalence classes.  Observe that each element of $\F_2$ naturally encodes a set and its complement: $E(X)\in \F_2$ denotes the equivalence relation consisting of exactly two classes, $X$ and $\overline{X}$.
\item An equivalence relation $R$ is \emph{light} if $\Id\leqc R$. It is easy to see that the light equivalence relations are exactly the infinite equivalence relations which have a computable \emph{transversal}, i.e., a computable sequence $\set{x_i}_{i\in\omega}$ of pairwise nonequivalent numbers;
\item An equivalence relation $R$ is \emph{dark} if $R$ is infinite and $\Id\not\leqc R$. 
\item For each of these {families}, the boldface version represents the collection of $\ER$-degrees containing members of the class. For example, $\bF$ is the set of degrees of finite equivalence relations, $\bDark$ is the set of degrees of dark equivalence relations, etc.
\end{itemize}

As is clear from the above, $\ER$ is partitioned into $\bF$, $\bLight$, and $\bDark$. Moreover, $\bI\subseteq \bF$. Inside $\ER$, computable equivalence relations can be readily characterized.

\begin{obs}[\cite{gao2001computably}, Prop.~3.3 and 3.4]\label{ceers initial}
	The degrees of computable equivalence relations form an initial segment of $\mathbf{ER}$ of order type $\omega+1$, and are exactly $\bI\cup\{\bId\}$.
\end{obs}

\begin{proof}
First, note that
\[
\Id_1 < \ldots \Id_n < \Id_{n+1} < \ldots \Id.
\]
So, the family $\mathcal{I}\cup \set{{\Id}}$ of equivalence relations has order type $\omega +1$. 

Let $R$ be a computable equivalence relation. Then the set 
\[
S:=\{x : \min [x]_R=x\}
\]
is computable. Let $S=\{c_0<c_1<c_2\ldots \}$. Then the function which sends each $[c_i]_R$ to $i$ is a computable function giving a reduction of $R$ to $\Id_{\abs{\omega_R}}$ (letting $\Id_\omega=\Id$). Further, this function is onto the classes of $\Id_{\abs{\omega_R}}$ and the inverse function on classes is also computable, so $R\equivc \Id_{\abs{\omega_R}}$.
%
%Secondly, observe that every computable equivalence relation $R$ is reducible to $\Id$ via the following function
%\begin{itemize}
%\item  $f(0)=0$,
%\item $f(x+1)= \begin{cases} f(y) &\text{if $(\exists y \leq x)(y\rel{R} x+1),$}\\
%\min z \notin [\range(f)\restriction x]_R &\text{otherwise}.
%\end{cases}
%$
%\end{itemize}
% Moreover, if $R$ has $n$ equivalence classes, then $R\equiv \Id_n$. 
%Hence, the computable equivalence relations coincide with $\mathcal{I}\cup \set{\textbf{Id}}$. 
\end{proof}

 The following is an easy, but useful fact about taking a uniform join with $\Id_1$, and how it essentially ``cancels out'' collapsing a computable class with another class.

\begin{lem}\label{Collapsing a computable class is subtracing id1}
	If $E$ is an \er with a computable class $C$, and $B$ is any other $E$-class, then $E_{/(\min C, \min B)}\oplus \Id_1 \equivc E$.
\end{lem}
\begin{proof}
	To show $E_{/(\min C, \min B)}\oplus \Id_1 \leqc E$, let $f\colon  E_{/(\min C, \min B)}\leqc E$ be defined by sending every element of $C$ to $\min B$ and be the identity on $\overline C$. Then notice that the class of $C$ is avoided by $f$. This lets us extend $f$ to a reduction of $E_{/(\min C, \min B)}\oplus \Id_1 \leqc E$ by sending the $\Id_1$-class to the class $C$ in $E$.
	The function $g(x)=2x$ for every $x\notin C$ and $g(x)=1$ for $x\in C$ gives a reduction $g:E\leqc E_{/(\min C, \min B)}\oplus \Id_1$.
\end{proof}

Note that $\Ceers$, $\bF$, and $\bigcup_{i\leq n} \bF_i$ for each $n$ are each initial segments of $\mathbf{ER}$. An obvious elementary difference between $\Ceers$ and $\ER$ is that the former degree structure is bounded and the latter is not.

\begin{obs}\label{obs:upward density}
$\ER$ has a least element, but no maximal element.
\end{obs}

\begin{proof}
Every constant function computably reduces $\Id_1$ to any given \er. Hence, $\bId_1$ is the least degree of $\ER$. On the other hand, for a given $R$, let $X$ be $\deg_T(R)$ and consider $E(X')$. We have that $E(X')\nleqc R$, as otherwise $X'$ would be $\leq_m R$ by Lemma \ref{f respects m-degrees}, but $R$ is strictly Turing below $X'$. So, $R<_c R\oplus E(X')$ and $R$ is not maximal.
\end{proof}

We now turn to some facts about dark \ers. The next two lemmas are adapted from the setting of ceers  \cite[Lemmas~4.6 and 4.7]{andrews2019joins}. The proof is essentially the same.

\begin{lemma}\label{dark implies self-full}
Dark equivalence relations are self-full. 
\end{lemma}

\begin{proof}
Let $R$ be dark. Suppose that there is  $f\colon  R \leqc R$ which avoids a given  $X\in\omega_R$. Let $x \in X$ and consider $\orb_f(x)$. From the fact that $f$ is a self-reduction of $R$  and $X\notin \range(f^\star)$, it follows that $\orb_f(x)$ is a c.e.\ infinite transversal of $R$, contradicting the darkness of $R$.
\end{proof}

\begin{lemma}\label{dark not reduces to proper quotient}
If $R$ is dark, then $R$ is not reducible to any of its proper quotients.
\end{lemma}

\begin{proof}
%Assume that there is $f\colon  R\leq_c R$ which avoids some $X\in\omega_R$. Let $Y\in\omega_R$ be an equivalence class different from $X$ and consider the  quotient $R_{/X\times Y}$ obtained by collapsing $X$ and $Y$ in $R$. Since $f$ avoids $X$, we have that $f\colon  R\leq_c R_{X\times Y}$. This shows that if $R$ is not reducible to any of its proper quotients, then $R$ is self-full.
Towards a contradiction, suppose that a dark $R$ is reducible to one of its proper quotients $R_{/A}$, via some $f$. Note that, since $R$ is dark, $R_{/A}$ must be infinite. Now, let $X,Y \in\omega_R$ be two equivalence classes that are collapsed in $R_{/A}$ and choose $x\in X$ and $y\in Y$. We claim that at least one of $\orb_f{(x)}$ or $\orb_f{(y)}$ cannot intersect $X\cup Y$. 
%FROM URI: Notation problem: $\orb$ includes the $0$th iterate of $f$, so $x\in \orb_f(x)$ and $y\in \orb_f(y)$. Probably we can just remove that $0$th iterate of $f$ from the definition of $\orb$?
Indeed, suppose that $i,j>0$ are minimal so that $\set{f^{(i)}(x), f^{(j)}(y)}\subseteq X\cup Y$, and, without loss of generality, suppose $i\geq j$. Since $X$ and $Y$ are collapsed in $R_{/A}$, we  have that  $f^{(i)}(x)\rel{R_{/A}}f^{j}(y)$. But since $f\colon  R\leqc R_{/A}$ and $R_{/A}\supseteq R$, this would imply that $f^{(i-j)}(x)\rel{R}y$, which either contradicts $x \rel{\cancel{R}} y$, if $i=j$, or contradicts the minimality of $i$, if $i>j$.

%
% $f\colon  R\leq_c R_{/A}$, and $R_{/A}\supseteq R$, we have that  $(f^{(i)}(x),f^{j}(y))\in R_{/A}$ implies $(f^{(i-j)}(x),y)\in R$, which either contradicts $x \cancel{R} y$ (if $i=j$), or contradicts the minimality of $i$ (if $i>j$).

So, one can assume that $\orb_f(x)\cap (X\cup Y)=\emptyset$. Now suppose that, for $i> j$, $f^{(i)}(x)\rel{R} f^{(j)}(x)$. Reasoning as above, we obtain that $f^{(i-j)}(x)\rel{R} x$, a contradiction. Hence, $\orb_f(x)$ would be a c.e.\ transversal of $R$. But this contradicts the darkness of $R$.
%Since $f\colon  R\leqc R_{/A}$, it must be the case that either $\orb_f{\min X}\cap X\cup Y=\emptyset$ or $\orb_f{\min Y}\cap X\cup Y=\emptyset$  reduces It must be the case that 
%On the other hand, assume that $R$ is self-full and, towards a contradiction, suppose that there is $A$ such that 
\end{proof}

We now introduce the dark minimal equivalence relations.
\begin{defn}
	An equivalence relation $R$ is \emph{dark minimal} if it is dark and its degree is minimal over $\bF$, i.e., if $S<_c R$ then $S$ is finite.
\end{defn}

%\luca{As recommended by the second reviewer, I have moved the following lines about the existence of dark minimal equivalence relations before the lemmas about their properties}.

Dark minimal equivalence relations exist (see \cite[Theorem~4.10]{andrews2019joins} for examples of dark minimal ceers) and they will occur several times in this paper, as their combinatorial properties will facilitate our study of the logical complexity of $\ER$. We conclude the preliminaries highlighting a couple of fundamental features of dark minimal equivalence relations.

\begin{lem}\label{lem:dark minimal and ce sets}
	Let $R$ be a dark minimal equivalence relation. Let $W$ be a c.e.\ set which intersects infinitely many $R$-classes. Then $W$ must intersect every $R$-class.
\end{lem}
\begin{proof}

Suppose $W$ intersects infinitely many $R$-classes. Consider the equivalence relation $R\restriction W$ and note that $R\restriction W\equivc R$ since $R\restriction W$ is not in $\F$ and $R$ is minimal over $\F$. Thus, we have reductions $R\leq R\restriction W\leq R$ with the second reduction given by inclusion. Since $R$ is dark, it is self-full by Lemma \ref{dark implies self-full}, so the reduction of $R$ to itself through $R\restriction W$ must hit every $R$-class. In particular, $W$ must intersect every $R$ class.
%
%Note that $R\restriction W <_c R\restriction W \oplus \Id_1 \leq R$. The strict inequality follows from the darkness of $R$ and Lemma \ref{dark implies self-full}, and the second reduction is by extending the inclusion of $R\restriction W$ to $R$ by sending $\Id_1$ to $y$. But $R\restriction C$ has infinitely many classes, so this contradicts the minimality of $R$ over $\F$.
\end{proof}

For the next lemma, recall that two sets of natural numbers $A,B$ are \emph{computably separable} if there is a computable set $C$ such that $A \subseteq C$ and $C \cap B=\emptyset$.

\begin{lemma}\label{dark minimal implies r.i.}
		Let $R$ be a dark minimal equivalence relation. Then the elements of $\omega_R$ are pairwise computably inseparable.
\end{lemma}
\begin{proof}
	Let $C$ be any computable set. Either $C$ or $\omega\smallsetminus C$ intersects infinitely many $R$-classes. Thus by Lemma \ref{lem:dark minimal and ce sets}, either $C$ or $\omega\smallsetminus C$ intersects every $R$-class, so $C$ cannot separate two $R$-classes.
\end{proof}

\section{Definability in $\ER$ and existence of least upper bounds}\label{sec:definability}

A natural way of understanding the logical complexity of a structure is by exploring which of its fragments are definable. In this section, we show that many natural families of equivalence relations are first-order definable without parameters.

\subsection{Defining the class of finite equivalence relations} In the case of ceers, the equivalence relations with finitely many equivalence classes are easily characterized: A ceer $R$ has $n$ equivalence classes 
if and only if $R \equivc \Id_n$. Hence in $\Ceers$, $\bF$ coincides with $\bI$ (and therefore it has order type $\omega$). These form an initial segment of $\Ceers$ and they are definable as the collection of nonuniversal ceers which are comparable to every ceer.

In $\mathbf{ER}$, the picture is much more delicate.
 %in fact  Theorem \ref{F2secondorderarithm}  below says that every countable distributive lattice embeds as an initial segment in $\F_2$. 
 For the moment, just observe that $\mathcal{F}\not\subseteq \I$: to see this, take   $E(X)$ with $X$ noncomputable. Moreover, while $\Id$ bounds $\I$, no equivalence relation can bound $\F$ (see the proof of Observation \ref{obs:upward density}). 

 We will show that $\bI$ is definable in $\ER$ as the collection of degrees which have  a least upper bound with any other degree, and from  that definition  will easily follow that $\bF$ is also definable. To obtain this result, throughout this section we will focus on the existence of least upper bounds of \ers, obtaining several structural results of independent interest.

The following lemma describes the shape of a potential least upper bound of \ers. An upper bound $T$ of equivalence relations $R,S$ is \emph{minimal} if there is no upper bound $V$ of $R,S$ such that $V <_c T$.

\begin{lemma}\label{normal form for joins}
Suppose $f\colon  R\leqc T$ and $g\colon  S\leqc T$. Then there is a pure quotient $U$ of $R\oplus S$ and reductions $f_0\colon R\leqc U$ given by $f_0(x)=2x$ and $g_0\colon  S\leqc U$ given by $g_0(x)=2x+1$ and $h\colon  U\leqc T$ so that $f=h\circ f_0$ and $g=h\circ g_0$. 

In particular, if $T$ is a minimal upper bound of \ers $R$ and $S$, then $T$ is equivalent to a pure quotient of $R\oplus S$.
\end{lemma}
\begin{proof}
Let $f\colon  R\leqc T$, $g\colon S\leqc T$, and $A:=\set{(2x,2y+1): f(x) \rel{T} g(y)}$. Then $R\oplus S_{/ A}$ is a pure quotient of $R\oplus S$.
% Indeed, suppose that there are $[u]_R\neq [v]_R$ such that $2u \rel{R\oplus S_{/A}} 2v$. By definition of $A$, we have $f(u) \rel{T} f(v)$, contradicting the fact that $f$ is reduction and $u\rel{\cancel{R}} v$.
Now, observe that $R\oplus S_{/ A}\leqc T$ via the function $h = f\oplus g$. And observe that $f = h\circ (x\mapsto 2x)$ and $g= h\circ (x\mapsto 2x+1)$.
%Let $x\in R$ be $\sim$-related to $y\in S$ if and only if $f(x)Tg(y)$. It is clear that $R\oplus S_{/\sim}$ is a good quotient of $R\oplus S$. 
% Hence, if $T$ is a minimal upper bound of $R$ and $S$, we obtain $R\oplus S_{/ A} \equivc T$.
\end{proof}

In $\ER$, to have a least upper bound is a rather strong property. The next result will state that any pair of \ers which are sufficiently incomparable cannot have a least upper bound.

%
%
%$I$ (in $\Ceers$). In $\mathbf{ER}$ the situation is more 
%
%We initiate our study of $\mathbf{ER}$ by focusing on equivalence relations with finitelu many classes. 
%
%
% $\mathcal{F}$, the initial segment consisting of the equivalence relations with finitely many classes. For every $k$, we denote by $\mathcal{F}_k$ the family of the equivalence relations with exactly $k$ equivalence classes. 
%
%
%
%
%
%
%
%\section{On the existence of joins}

%Of particular use in defining subclasses in $\ER$ will be examining when pairs of \ers have least upper bounds. We begin by seeing that any pair of \ers that are sufficiently different cannot have a least upper bound.

\begin{defn}
	Define $R\leq_F S$, if there is computable set $A$ so that $R\restriction  A\leqc S$ and $R\restriction \overline{A}$ is finite. 
\end{defn}
Obviously, $\leqc$-reducibility implies $\leq_F$-reducibility. The converse does not hold as there are $\leqc$-incomparable $X,Y\in \F$, but $X\equiv_F Y$.  

\begin{thm}\label{F incomparable no join}
	If $R$ and $S$ are \ers which are $\leq_F$-incomparable, then $R$ and $S$ do not have a least upper bound in $\ER$.
\end{thm}
\begin{proof}
Suppose towards a contradiction that $T$ is the least upper bound for $R$ and $S$. By Lemma \ref{normal form for joins}, we can assume that $T$ is a pure quotient of $R\oplus S$. We will build by stages another pure quotient $V (=\bigcup V_s)$ of $R\oplus S$ such that $T \nleq V$, contradicting the supposition. To do so, we let $V_0$ be $R \oplus S$ and, at further stages, we will collapse $R$-classes and $S$-classes in $V$ to diagonalize against all potential reductions from $T$ to $V$. We note that we are constructing $V$ to be c.e. in the Turing degree $\deg_T(R)\vee \deg_T(S)\vee \deg_T(T)\vee \mathbf{0}''$.

\subsubsection*{The construction} 
At stage $s$, we may \emph{restrain} some $V_s$-classes so that, at the end of the construction, they will be V-classes.
When we say that numbers are  restrained, we mean that they come from restrained classes.
\subsubsection*{Stage $0$} Let $V_0:= R\oplus S$. Do not restrain any  equivalence class.
 
\subsubsection*{Stage $e+1$} If $\phi_e$ is nontotal, let $V_{e+1}:=V_e$. Otherwise, search for a pair of distinct numbers $(u,v)$ such that $\phi_e(u)\downarrow=x_e,$ $\phi_e(v)\downarrow=y_e$, and
\begin{enumerate}
\item[$(a)$] either $u\rel{T} v \notleftright x_e \rel{V_e} y_e$,
\item[$(b)$] or $u\rel{\cancel{T}} v$ and $x_e$ and $y_e$ have different parity and they are both unrestrained.
\end{enumerate} 

\smallskip

We will show in Claim~\ref{claim:terminate} below  that such a pair will always be found. If the outcome is $(a)$, let $V_{e+1}:=V_e$ and restrain the $V_e$-classes of $x_e$ and $y_e$. If
 the outcome is $(b)$, let $V_{e+1}:={V_e}_{/(x_e,y_e)}$ and we restrain the common $V_{e+1}$-class of $x_e$ and $y_e$. 

\subsubsection*{The verification} The verification relies on the following claim.

\begin{claim}\label{claim:terminate}
The action defined at stage $e+1$ (i.e., the search of a pair of numbers satisfying either $(a)$ or $(b)$) always terminates.
\end{claim}

\begin{proof}
Suppose that there is a stage $e+1$ at which no pair $(u,v)$ is found. This
means that $\phi_e$ is total and  $\phi_e\colon  T\leqc V_e$;  otherwise, we would reach outcome $(a)$. Next, observe that $\phi_e$ cannot hit infinitely
many equivalence classes of both $V_e \restriction \evens$ and $V_e \restriction \odds$; otherwise, since only finitely many equivalence classes are restrained at each stage and $V_e$ coincides with $R\oplus S_{/A}$ for a finite set $A$, there would be a pair of numbers of different parity which are unrestrained and  we would reach outcome $(b)$. 

So, without loss of generality, assume
that $\phi_e$ hits only finitely many classes in $V_e \restriction \odds$. Let $f$ be the following partial computable function,
\[
f(x)= \begin{cases}
\dfrac{\phi_e(x)}{2} &\text{$\phi_e(x)$ is even,}\\
\uparrow   &\text{otherwise.}
\end{cases}
\]
We have that $f\colon T \restriction \dom(f) \leqc R$ and $T \restriction \overline{\dom(f)}$ is finite. Thus, $T \leq_F R$. Since $S\leqc T$ and $T \leq_F R$, we obtain that $S \leq_F R$, which contradicts the fact that $R$ and $S$ are $\leq_F$-incomparable.
\end{proof}

It follows from the above construction  that $V$ is a pure quotient of $R\oplus S$. In particular, every time we collapse an odd with an even class, we restrain all members of that class, so it cannot be part of any future collapse. Hence, by Lemma \ref{obs:pure quotients are ubs}, $V$ is an upper bound of $R$ and $S$. Towards a contradiction, assume that $T \leqc V$ via some $\phi_i$. Claim  \ref{claim:terminate} ensures that the action defined at stage $i+1$ terminates with either disproving that $\phi_i$ is a reduction from $T$ to $V_i$ or by
providing two equivalence classes that will be $V$-collapsed to diagonalize against $\phi_i$. That is, the restraints and the fact that $V$ is a quotient of $V_i$ guarantee that $\phi_i \colon  T \not\leqc V$, a contradiction.
\end{proof}

%\begin{proof}
%	Suppose towards a contradiction that $T$ is a least upper bound for $E$ and $R$. By observation \ref{normal form for joins}, we can assume without  We build an \er $X$ so that $X\restriction \evens = E$, $X\restriction \odds = R$. We let $e_1,e_2,\ldots$ enumerate the classes of $X\restriction \evens$ and $r_1,r_2,\ldots$ enumerate the classes of $X\restriction \odds$. Let $\sigma$ be some sufficiently generic permutation of $\omega$. In $X$, we further collapse $e_i$ with $r_{\sigma(i)}$ for each $i$.
%	
%	We now want to show that $T\not\leq_c X$, so suppose that $f$ is a reduction of $T$ to $X$.  Suppose that infinitely many of the sets $e_i$ contain an element of $\im(f)$ and infinitely many of the sets $r_j$ contain an element of $\im(f)$. Then since $\sigma$ is sufficiently generic, there will be some $e_i$ containing $f(x)$ and $r_{j}$ containing $f(y)$ so that $x\rel{T}y$ if and only if $\sigma(i)\neq j$. That is, the genericity of $\sigma$ ensures that $f$ is not a reduction of $T$ to $X$. Thus, we can suppose that there are only finitely many $e_i$ in the image of $f$. Let $g$ be the composition of the reduction of $E$ to $T$ and $T$ to $X$. Let $A$ be the set of elements $x\in \omega$ so that $g(x)$ is odd. Then $A$ is computable and $g$ is a reduction of $E\restriction A$ to $R$. Further, the image of $g\restriction \omega\setminus A$ hits only finitely many classes in $X$. It follows that $E\restriction\omega\setminus X$ has only finitely many classes. Thus $E\leq_F R$ contrary to assumption.
%\end{proof}

\begin{cory}\label{cor:dark no lub with id}
	No dark \er has a least upper bound with $\Id$.
\end{cory}
\begin{proof}
	It suffices to show that no dark \er $R$ can be $F$-comparable with $\Id$. On the one hand, note that $\Id\restriction A\equiv\Id$ for any cofinite $A$ and thus $\Id\restriction A \not\leqc R$ since $R$ is dark. Therefore $\Id \nleq_F R$.
	On the other hand, suppose $R\restriction A\leqc \Id$ with $R\restriction \overline{A}$ finite. Observe that $R\restriction A \nequivc \Id$, as otherwise $R$ would be light because $\Id\leqc R \restriction A\leqc R$. So, $R\restriction A<_c\Id$ and, by Observation \ref{ceers initial}, this means that $R\restriction A$ is finite. As  $R\restriction \overline{A}$ is also finite, it follows that $R$ is finite, contradicting its darkness.
\end{proof}

Obviously, if $R\in\mathcal{F}$, then $R$ is $F$-reducible to any given \er $S$. This property does not guarantee that $R$ has a least upper bound with every other equivalence relation (see Theorem~\ref{NCA} below). The next lemma says that the finite \ers are the  only ones which are $F$-comparable with any other \er.

\begin{lem}
	If $R$ is infinite, then there is an infinite $S$ so that $R$ and $S$ are $\leq_F$-incomparable.
\end{lem}
\begin{proof}
	If $R$ is dark, let $S$ be $\Id$ and use Corollary \ref{cor:dark no lub with id}.
	If $R$ is light, then let $S$ be any dark \er such that $\{Y\in \omega_S : Y\nleq_m R \}$ is infinite. We will show that such $S$ exists after verifying its $\leq_F$-incomparability with $R$.
	
	Note that if  $R\restriction \overline{A}$ is finite, then $R\restriction A$ must be light, because $R$ is  light. It follows that $R\not\leq_F S$. Next, let $A$ be so that $S\restriction \overline{A}$ is finite. There exists $[y]_S\subseteq A$  which is not $\leq_m R$. But, by Lemma \ref{f respects m-degrees} this shows that $S\restriction A\not\leqc R$.
	
	To see that such an $S$ exists, we begin with any dark ceer $S_0$ and we partition $\omega_{S_0}$ into infinitely many infinite families $\mathcal{M}_i$. Next, define
	\[
	\mathcal{N}_i:=\left\{ \bigcup_{I \in \mathcal{J}} I: \mbox{ for } \mathcal{J}\subseteq \mathcal{M}_i	
 \right\}.
	\]
Each $\mathcal{N}_i$ is obviously uncountable and so it contains a set $X_i$ whose $m$-degree does not reduce to the  degree of $R$. Let $S$ be a quotient of $S_0$ such that $X_i\in\omega_S$, for all $i$. Since a quotient of a dark \er is dark, this $S$ satisfies our requirements.
\end{proof}

Combining the last lemma with Theorem \ref{F incomparable no join}, we immediately obtain the following.

\begin{cory}\label{infNotInI}
	If $R$ is  infinite, then there is an infinite  $S$ so that $R$ fails to have a least upper bound with $S$.
\end{cory}

It might seem at this point that any pair of degrees ought to not have a least upper bound, but we now show that there are pairs of infinite degrees which have a least upper bound.

\begin{thm}\label{thm:joins can exist}
	There are incomparable \ers $R,S\notin \mathcal{F}$ which have a least upper bound.
\end{thm}
\begin{proof}
	
	Let $R_0$ be a dark \er with all computable classes (the existence of such equivalence relation follows from, e.g., \cite[Proposition~5.6]{gao2001computably}) and let $S_0\in \F\smallsetminus \I$.
	%\uri{Maybe get incomparability immediately by choosing $S_0\not\leqc R_0\oplus \Id$.}
	 Let $R:=R_0\oplus \Id$ and $S:=S_0\oplus \Id$. First, we note that $R$ and $S$ are incomparable. Indeed, on the one hand, there must be a noncomputable $S$-class, since  $S_0\notin \I$. By Lemma~\ref{f respects m-degrees}, this suffices to guarantee that $S\not\leqc R$, as all $R$-classes are computable. On the other hand, suppose that $i: R \leqc S$. Then, the following c.e.\ set
\[
\{x : i(2x) \in \odds \mbox{ and, for all $y<x$, }  i(2x)\neq i(2y)\}
\]	
	would be an infinite transversal of $R_0$, contradicting its darkness.
	
Next, we will prove that $R\oplus S\equivc R_0\oplus S_0\oplus \Id$ is a least upper bound of $R$ and $S$. Let $U$ be any \er with reductions $f_0\colon R\leqc U$ and $g\colon S\leqc U$. We claim that there is $f_1 : R\leqc U$ whose image is disjoint from the image of the $S_0$-classes given by $g$.
%	, i.e.,
%\[
%(\exists f_1)(\range(f_1)\cap\range(g\restriction\evens) =\emptyset).
%\]	
To prove this, define the following collection of equivalence classes of $R$
\[
\mathcal{C}:=\{X\in\omega_R : {f_0^\star}(X)\cap \range(g\restriction\evens) \neq\emptyset\}.
\]
Let $\mathcal{C}$ be $\{C_0,\ldots,C_k\}$ ($\mathcal{C}$ is finite since $S_0\in \mathcal{F}$). Moreover, note that all elements of $\mathcal{C}$ are computable. Let $m=\max(\{\min(C_i): i\leq k\})$.

The following function is  constructed from $f_0$ by suitably shifting the elements in $\Id$ to avoid the finite overlap with $g\restriction \evens$:
\[
f_1(x):=\begin{cases}
f_0(2(u+k+m)+1) &\text{$(\exists u)(x=2u+1)$}\\
f_0(2(i+m)+1) &\text{$(\exists i\leq k)(x\in C_i)$} \\
f_0(x) &\text{otherwise}.\\
\end{cases}
\]	
It is straightforward to check that $f_1$ is a computable reduction from $R$ to $U$ which satisfies the  property that $\range(f_1)\cap\range(g\restriction\evens)$ is empty.
It is exactly this property which allows to combine $g$ and $f_1$ in a natural way so as to obtain the desired reduction from $R_0\oplus S_0\oplus \Id$ to $U$:
\[
h(x):=\begin{cases}
f_1(2u) &\text{$(\exists u)(x=3u)$}\\
g(2u) &\text{$(\exists u)(x=3u+1)$} \\
f_1(2u+1) &\text{$(\exists u)(x=3u+2)$}.\\
\end{cases}
\]	
This concludes the proof.
%\lucaout{By Lemma \ref{normal form for joins}, there is a pure quotient $T_0$ of \lucainsert{$R_0\oplus S_0$} \lucaout{$R_0$ and $S_0$} so that the reductions $f\restriction \evens$ and $g\restriction \evens$ filter through $T_0$. Then consider the function $f\restriction \odds$. It is immediate that the range of $f\restriction \odds$ intersects no class in the range of $f\restriction \evens$, and it can only finitely intersect the classes in the range of $g\restriction \evens$. So, putting together $f$ (after shifting the elements in $\Id$ to avoid the finite overlap with $g\restriction \evens$) and $g\restriction \evens$, we get a reduction of $T_0\oplus \Id$ to $U$. Since every class of $R_0$ is computable, any pure quotient $T_0$ over $R_0$ and $S_0$ has the property that $R_0\oplus S_0\equivc T_0\oplus \Id_k$ for some $k$ by Lemma \ref{Collapsing a computable class is subtracing id1}. Thus $T_0\oplus \Id\equivc R_0\oplus S_0\oplus \Id$. Thus we have a reduction of $R_0\oplus S_0\oplus \Id$ to $U$.}
\end{proof}

%\begin{remark}
%In \cite{}, the authors asked for which levels of the Ershov hierarchy, there are equivalence relations lying properly at that level at  having sup at that level.	If we choose a finite \er in $\Delta^0_2$ and use a dark \er instead of $\Id$ in the argument, this shows that there are some.
%\end{remark}

%
%\smallskip
%
%\subsection{Defining $\I$ and $\F$}

We are now in position to show that $\bI$ is definable. 
%$We will show that an \er is in $\I$ if and only if it has a join with every other \er. We first show that if $E\in \I$, then $E$ has a least upper bound with every \er. Next we examine the collection of \ers which have joins with every finite \er.

\begin{thm}\label{thm: I is definable}
	$\bI$ is definable in $\ER$ as the collection of degrees which have least upper bounds with every other degree.
\end{thm}
\begin{proof}
	
	We first verify that every member of $\I$ has a least upper bound with every other \er.

\begin{lem}
	If $E\in \I$, then $E$ has a least upper bound with any \er $R$.
\end{lem}
\begin{proof}
	Let $E=\Id_k$. If $R$ has at least $k$ classes, then $R$ is the least upper bound of $\Id_k$ and $R$, as $\Id_k\leqc R$. Otherwise, let $n<k$ be $|\omega_R|$. We prove that $R\oplus \Id_{k-n}$ is the least upper bound of $\Id_k$ and $R$. 
First, it is immediate that both $R$ and $ \Id_k$ reduce to $R\oplus \Id_{k-n}$. Next, suppose that $R$ and $\Id_k$ are reducible to some $S$ and let $f\colon  R\leqc S$. Then, $f$ can only hit $n$ equivalence classes of $S$, but $|\omega_S|\geq k$ because $\Id_k \leqc S$. Let $A=\set{a_1,\ldots, a_{k-n}}$ be a set of representatives from $k-n$  equivalence classes which  $f$ avoids. By letting $g$ agree with $f$ on elements from $R$ and send the classes of $\Id_{k-n}$ to the numbers in $A$, we get a reduction $g\colon  R\oplus \Id_{k-n} \leqc S$.
%Let  $g: \Id_{k-n} \leqc \Id \restriction A$. Then $f\oplus g$ reduces $R\oplus \Id_{k-n}$ to $S$.
\end{proof}

Corollary 
\ref{infNotInI} 
guarantees that no infinite equivalence relation can have a least upper bound with every other \er. So, to prove the theorem, it suffices to show that the same is true for any  finite equivalence relation which is noncomputable.

%FROM URI: The previous version of this Lemma let $S$ have 2 classes, but this ignored the possibility that $R$ has computable classes which might be able to collapse to these 2 classes. I replaced it by making $S$ have $k$-classes, therefore even non-computable $R$-classes are forced to collapse with $S$-classes. Of course, this also just follows from the NCA argument below, since obviously no element of $\F\smallsetminus \I$ is NCA. There might be a way to re-order so we can just reference that?
We note that the following lemma also follows from Theorem \ref{NCA} below, but we include a proof here for self-containment of this section.
%FROM URI SHould we move the NCA section up to be the first thing? Then we could use it here? Seems like overkill to me, so I'm not doing so now. 

\begin{lem}
	If $R\in \F\setminus \I$, then there is $S\in \F\setminus \I$ so that $R$ and $S$ do not have a least upper bound.
\end{lem}

\begin{proof}
	Let  $|\omega_R|=k$, and since $R\notin \I$, fix $C$ to be a noncomputable $R$-class. Let $\omega = X_1\cup \cdots \cup X_k$ be a partition of $\omega$ so that each $X_i$ is $m$-incomparable with all noncomputable $Y\in \omega_R$. Next, let $S$ be the \er with classes $X_i$ for $i\leq k$.  
	% Let $a_1\cancel{R}a_2$ and, without loss of generality, assume that $0\in X$ and $1\notin X$. 
	%Finally, let $S$ be $E(X)$, i.e., the equivalence relation with exactly two classes: $X$ and and its complement.
	 Towards a contradiction, suppose that $T$ is a least upper bound of $R$ and $S$. We may assume that $T$ is a pure quotient $R\oplus S_{/A}$ by Lemma \ref{normal form for joins}.
	
	First, observe that $T$ has exactly $k$ classes: if there were fewer, then  $R\nleqc T$; if there were more, then we can take $Z$ to be a pure quotient of $R \oplus S$ which has exactly $k$ classes and we would have $T\not\leqc Z$. Thus $C$ is collapsed via $A$ with some class $X_i$ in $T$.
	%(e.g., $R\oplus S/_{\{(2\cdot a_1,1),(2\cdot a_2,3)\}}$). 

	% Consider $E(X)$, i.e., the equivalence relation with two classes: $X$ and $\omega\setminus X$. Suppose towards a contradiction that $E$ is a least upper bound of $F_1$ and $F_2$.
	
%	Let $a_1,a_2$ be in two different $F_1$-classes, and suppose $0\in X$ and $1\notin X$ Let $R=F_1\oplus F_2/\{(2\cdot a_1,1),(2,a_2,3)\}$. That is, we take the uniform join of $F_1$ and $F_2$ and we collapse the class of $X$ in $F_2$ with the class of $a_1$ in $F_1$ and the complement of $X$ in $F_2$ with the class of $a_2$ in $F_1$. This is an \er with the same number of classes as $F$ yet is an upper bound for $F_1$ and $F_2$. Thus $E$ must have $k$ classes. 
	
	Now, let $f\colon T\leqc R\oplus S$,
	  and consider the image of $C$ in the composed reduction $R\leqc R\oplus S_{/A}\leqc R\oplus S$. Since $C\not\leq_m X_j$ for any $j\leq k$, the image must be contained in the evens.
	 Similarly, consider the image of $X_i$ under the composed reduction $S\leqc R\oplus S_{/A}\leqc R\oplus S$. Since $X_i\not\leq_m K$ for any $K\in \omega_R$, the image must be contained in the odds. But $C$ and $X_i$ are $A$-collapsed in $T$, which contradicts $f$ being a reduction.
\end{proof}

This completes the proof of Theorem \ref{thm: I is definable}.
\end{proof}

The next corollary immediately follows from the definability of $\bI$.

\begin{cory}\label{definable cory} For all $k$,
	\begin{itemize}
		\item $\bId_k$ is definable as the unique degree in $\bI$ which has exactly $k-1$ predecessors;
		\item $\bF_k$ is definable in $\ER$ as the degrees which bound $\bId_k$ and not $\bId_{k+1}$;
		\item $\bF$ is definable in $\ER$ as the degrees which do not bound every member of $\bI$.
	\end{itemize}
\end{cory}

\subsection{Defining the identity}\label{sec:NCA}

In this section, we give a combinatorial characterization for the degrees which have least upper bounds with every member of $\bF$. We will then use this analysis to give a definition of the degree $\bId$ (and thus $\bLight$ and $\bDark$) in $\ER$ as a combination of its minimality over $\bF$ along with the property of having least upper bounds with every degree in $\bF$. 

%We note that neither condition alone suffices. The first because there are dark minimal ceers. The second because there  are infinitely many degrees which have a join with every
%member of $\F$, as follows from our characterization below. 

We will need the following combinatorial lemma:

\begin{lem}\label{avoid computable classes from a uniformly computable sequence of computable classes}
	Let $R$ be an \er with a uniformly computable sequence $(C_i)_{i\in \omega}$ of distinct computable $R$-classes. Let $S\subset \omega$ be a finite set. Then there is a reduction of $R$ to itself which avoids every $C_i$ for $i\in S$.
\end{lem}
\begin{proof}
	We construct the reduction $f\colon R\leqc R$ in stages. At every stage $s$, we will construct a partial function $f_s$ and a parameter $X_s$, which will be a finite subset of $\omega$. At stage $s+1$, we will ensure $f_{s+1}(s)$ is defined. 
	
	\subsubsection*{Stage $0$} Let $f_0=\emptyset$ and $X_0=S$.
	
	\subsubsection*{Stage $s+1$} We distinguish three cases.
	\smallskip
	
	\begin{enumerate}	 
	\item If $s\notin \bigcup_{n\in X_s} C_n$, let $f_{s+1}:=f_s\cup \{(s,s)\}$ and let $X_{s+1}:=X_s$. 
	
	\item $s\in C_n$ for some $n\in X_s$ and there is some $k<s$ in $C_n$. Then let $f_{s+1}:= f_s\cup \{(s,f(k))\}$ and $X_{s+1}:=X_s$.
	
	\item $s\in C_n$ for $n\in X_s$ and $s$ is $\min(C_n)$. Then let $m$ be least so that $m\notin X_s$ and $\ran f_s\cap C_m=\emptyset$. Let 
$f_{s+1}:=f_s\cup \{(s,\min C_m)\}$
	 and let $X_{s+1}:=X_s\cup m$.
\end{enumerate}
\smallskip

%\uri{Below we use $X_s$ seemingly for the union of the set of classes $C_i$ for $i\in X_s$. I'm replacing these $X_s$'s by $Y_s$'s as defined below.}

	We argue by induction that every $f_s$ is a partial reduction of $R$ to itself and no member of $C_n$, for $n\in S$, is in the range of $f_s$. For each $s$, let $Y_s=\bigcup_{i\in X_s}C_i$. We note that $f_s$ is the identity on $\overline{Y_s}$ and $\ran(f_s\restriction Y_s)\subseteq Y_s$, so we only need to show that $a\rel{R} b \leftrightarrow f_s(a)\rel{R} f_s(b)$ for $a,b\in Y_s$ and that no element of $Y_s$ is sent into a class $C_n$ for $n\in S$. Note that when a number $n$ first enters $X_k$ for $k<s$, then $C_n$ is neither in the domain nor range of $f_{k-1}$. Thus, for every $n\in X_s$, case $(2)$ ensures that each class is sent via $f$ to the same location. That is, $a \rel{R} b\rightarrow f(a)\rel{R} f(b)$ for $a,b\in Y_s$. In case $(3)$, we define $f$ for an element of a class $C_i$ with $i\in X_s$, and note that we send it to a class which is not in the range of $f_s$. Thus, if $a\rel{\cancel{R}} b$ then $f_s(a)\rel{\cancel{R}} f_s(b)$ for $a,b\in Y_s$. Similarly, note that in case $(3)$, we only send these new classes to classes $C_m$ for $m$ outside of $X_s$. In particular, $m\notin X_0$, so we never put $C_m$ for $m\in S$ into the range of $f_s$.
\end{proof}

The next lemma identifies a natural way in which a uniformly computable sequence of computable classes may arise.

\begin{lem}\label{computable class to a non-computable class gives a uniformly computable sequence of computable classes}
	Let $f\colon R\leqc R$ and let $C\in \omega_R$ be a computable $R$-class. Suppose that $f^\star(C)$ is not a computable $R$-class. Then either there is some $i\in \omega$ so that $f^{(i)}$ avoids $C$ or there is a uniformly computable sequence of distinct computable $R$-classes $(C_i)_{i\in \omega}$.
\end{lem}
\begin{proof}
	Suppose that there is no $i\in \omega$ so that $f^{(i)}$ avoids $C$, and let $C_i=\{x:f^{(i)}(x)\in C\}$. It is immediate that this is a uniformly computable sequence of computable classes. We need only verify that they are distinct. Suppose that $C_i=C_j$ with $i<j$. Further, suppose that $i$ is minimal for such an example. Then $i=0$, as otherwise, we would have $C_{i-1}=C_{j-1}$ since $f$ is a reduction of $R$ to $R$. Thus we have some $C_{j}=C_0=C$. But then $f^\star(C)=f^\star(C_j)=C_{j-1}$ is computable, contrary to hypothesis.
\end{proof}

Putting the previous two lemmas together, we get a general result about avoiding classes.
\begin{cory}\label{non-comp classes with comp preimages are avoidable}
	If $f\colon R\leqc R$ and $C\in \omega_R$ is a computable $R$-class so that $f^\star(C)$ is a noncomputable $R$-class $D$, then there is some reduction $g\colon R\leqc R$ so that $g$ avoids $C$. Thus also $f\circ g$ avoids $D$.
\end{cory}
\begin{proof}
	We have two cases from Lemma \ref{computable class to a non-computable class gives a uniformly computable sequence of computable classes}. The first possibility is that $g=f^{(i)}$ avoids $C$ for some $i$. The second possibility is that there is a uniformly computable sequence of distinct computable $R$-classes $(C_i)_{i\in \omega}$. Then we can apply Lemma \ref{avoid computable classes from a uniformly computable sequence of computable classes} to give a reduction $g$ which avoids $C$. Then $f\circ g$ avoids $D$.
\end{proof}

We now present the combinatorial condition which we will show is equivalent to having a least upper bound with every member of $\F$.

\begin{defn}
	An \er $R$ is \emph{\NCA} if, for every finite collection $\mathcal{C}$ of noncomputable equivalence classes of $R$, there is a reduction $f\colon  R\leqc R$  which avoids all the equivalence  classes in $\mathcal{C}$.
\end{defn}

First we observe that avoiding any one noncomputable class is equivalent to avoiding any finite set of noncomputable classes. 

\begin{lem}\label{avoiding one makes you nca}
	Let $R$ be an \er so that for any noncomputable class $C$, there is a reduction of $R$ to itself that avoids $C$. Then $R$ is \NCA.
\end{lem}
\begin{proof}
	We proceed by induction on $k$ to show that for any set of size $k$ of noncomputable classes, there is a reduction of $R$ to itself which avoids every class in the set. For $k=0$, the claim is trivial. For $k=1$, the claim follows from the  assumption about $R$.
	
	Next, for $k>1$, let $S=\{C_1,\ldots C_{k+1}\}$ be a collection of noncomputable classes. By inductive hypothesis, there is a reduction $f\colon R\leqc R$ which avoids $C_2,\ldots, C_{k+1}$. We consider three cases depending on what type of class is sent to $C_1$ via $f$: If there is no class sent to $C_1$ via $f$, then $f$ avoids every class in $S$. If there is a noncomputable class $X$ sent via $f$ to $C_1$, then by assumption there is a reduction $g\colon R\leqc R$ which avoids $X$. Then $f\circ g$ avoids every class in $S$. Lastly, if a computable class $X$ is sent via $f$ to $C_1$, then Corollary \ref{non-comp classes with comp preimages are avoidable} shows that there is a reduction $g\colon R\leqc R$ which avoids $X$. Then $f\circ g$ avoids every class in $S$.
\end{proof}

Next, we show that the property of noncomputable avoidance is  degree invariant.

\begin{obs}\label{NCA is a degree property}
If $R$ is noncomputably avoiding and $R\equivc S$, then $S$ is also noncomputably avoiding.
\end{obs}

\begin{proof}
Let $S$ be  equivalent to some noncomputably avoiding $R$ via $f\colon R \leqc S$ and $g\colon  S \leqc R$. Given any noncomputable $S$-class $C$, we need to build $h\colon  S\leqc S$ such that $h$ avoids $C$.

If $C\notin \ran(f^\star)$, then $f\circ g$ is a reduction of $S$ to itself which avoids $C$. So, let $K$ be an $R$-class so that $f^\star(K) = C$.  It suffices to find a reduction $\ell$ of $R$ to itself avoiding $K$. Once we have this, $h= f\circ \ell \circ g$ is a reduction of $S$ to itself avoiding $C$.

If $K$ is noncomputable, then we use the hypothesis that $R$ is \NCA to give the reduction $\ell$, and we are done. So, suppose $K$ is computable. Observe that $g\circ f\colon R\leqc R$ and ${(g\circ f)}^{\star}(K)$ is not computable because $C$ is not computable. Thus, we can apply Corollary~\ref{non-comp classes with comp preimages are avoidable} to get a reduction $\ell$ of $R$ to itself avoiding the class $K$.
%
%Let $h:R\leqc R$ avoid every non-computable member of $\mathcal{K}_S$. We now consider $X=\{\mu_{h}^{-1}(C) :  C\in \mathcal{K}_S\}$. If $X$ is empty, then we are done. Otherwise, note that $X$ can only contain computable classes. In particular, these are the $h$-preimages of the computable classes in $\mathcal{K}_S$. Observe that $g\circ f\circ h$ must send these to non-computable $R$-classes, since every member of $\mathcal{C}_S$ is non-computable. Then we can apply Lemmas \ref{computable class to a non-computable class gives a uniformly computable sequence of computable classes} and \ref{avoid computable classes from a uniformly computable sequence of computable classes} to get a reduction $j$ of $R$ to itself that avoids every member of $X$. Then 
%$h\circ j$ is a reduction of $R$ to itself avoiding every class in $\mathcal{K}_S$, as needed.
%%Let $\mathcal{C}_R \subseteq \omega_R$ be the the preimage under $\mu_g$ of $\mathcal{C}_S$. Since $R$ is noncomputably avoiding, there is  $h\colon  R\leqc R$  such that $h$ avoids $\mathcal{C}_R$. To conclude, it is enough to take as $g$ the function $f_2 \circ h \circ f_1$. 
\end{proof}

Noncomputably avoiding \ers exist. For instance, any \er having all computable classes (and note that there are dark \ers with this property, see e.g.\ \cite[Lemma~3.4]{fokina2018measuring} or \cite[Prop.~5.6]{gao2001computably}) is obviously noncomputably avoiding. A less trivial example is provided by the following observation. 

\begin{obs}
The degree of universal ceers is noncomputably avoiding.
\end{obs}

\begin{proof}
Let $U$ be a universal ceer. Let $V=U\oplus U$ and note that $V\equivc U$ since $V$ is also a ceer. Any noncomputable class $C$ is either contained in $\evens$ or $\odds$. So, we can reduce $V$ to the copy of $U$ on the $\odds$ or, respectively, $\evens$ of $V$. This gives a reduction of $V$ to itself avoiding the class $C$. Thus, $V$ is \NCA by Lemma \ref{avoiding one makes you nca} and $U$ is \NCA by Lemma \ref{NCA is a degree property}.
%Denote by $V$ the cylindrification of $U$, i.e.,
%\[
%\langle i,x\rangle \rel{V} \langle j,y\rangle \Leftrightarrow (i=j \mbox{ \& }  x\rel{U}y).
%\]
%
%$V$ is obviously a ceer and therefore $V\leqc U$. On the other hand, the map $x\mapsto \langle 0	,x\rangle$ defines a reduction from $U$ to $V$. So, it suffices to show that $V$ is noncomputably avoiding. Observe that each column of $V$ is a universal ceer. Thus,  for any given finite collection $\mathcal{C}$ of $V$-classes to be avoided, it is enough to self-reduce $V$ to one of the columns that are disjoint from all equivalence classes in $\mathcal{C}$. 
\end{proof}

We now give the main result of this section characterizing the degrees which have a least upper bound with every \er in $\F$.

\begin{thm}\label{computable classes have join with finite}\label{NCA}
	
	An \er $R$ is \NCA if and only if $R$ has a least upper bound with every \er in $\F$.

%	If $R$ is any \er with all classes being computable, then $R$ has a least upper bound with every member of $\F$.
\end{thm}
\begin{proof}
$(\Rightarrow)$ Let $R$ be \NCA. Fix $S\in \F$ and let $k=\abs{\omega_S}$. Fix $a_1,\ldots, a_k$ representing the $k$ distinct $S$-classes. Let $j\leq k$ be the minimum of $k$ and the number of computable $R$-classes, and fix $C_1,\ldots C_j$ to be computable $R$-classes. We will show that $X:=R\restriction \overline{ \bigcup_{i\leq j} C_i} \oplus S$ is a least upper bound for $R$ and $S$. First note that {$X$} is an upper bound for $R$ (and trivially $S$) via the function $f(x)= 2a_i+1$ if $x\in C_i$ for $i\leq j$ and otherwise $f(x)=2x$.

By Lemma \ref{normal form for joins}, it suffices to show that $X$ is reducible to any pure quotient $R\oplus S_{/A}$ of $R\oplus S$. 
Fix a pure quotient $R\oplus S_{/A}$. Let $h\colon R\leqc R$ be a reduction of $R$ to itself which avoids every noncomputable $R$-class which is $A$-collapsed with an $S$-class in $R\oplus S_{/A}$. Let $K_1,\ldots K_m$ enumerate the $R$-classes so that $h^\star(K_i)$ is $A$-collapsed with an $S$-class. Note that these all must be computable, and $m\leq j$. If any $K_{i_0}$ equals some $C_{i_1}$ for $i_0,i_1\leq m$, then reorder the $K$'s so that $i_0=i_1$.

Let $g$ be a reduction of $R$ to itself which swaps $K_i$ with $C_i$ for $i\leq m$. That is, 
\[
g(x)= \begin{cases}
x & x\notin \bigcup_{i\leq m} C_i\cup \bigcup_{i\leq m} K_i\\
\min K_i	& x\in C_i\\
\min C_i & x\in K_i.
\end{cases}
\]
Then all $R$-classes which are sent via $h\circ g$ to an $R$-class $A$-collapsed with an $S$-class are among the classes $C_i$ for $i\leq m$. Thus, taking the restriction of $h\circ g$ to the set $\overline{\bigcup_{i\leq j} C_i}$ gives a reduction $f$ of $R\restriction \overline{ \bigcup_{i\leq j} C_i}$ to $R$ which avoids every $R$-class which is $A$-collapsed with an $S$-class. Then we can make a reduction $f'$ of $R\restriction \overline{ \bigcup_{i\leq j} C_i}\oplus S$ to $R\oplus S_{/A}$ by following $f$ on $R\restriction \overline{ \bigcup_{i\leq j} C_i}$ and being the identity map on $S$-classes.

$(\Leftarrow)$ Assume that $R$ has a least upper bound with every finite equivalence relation, and fix a noncomputable class $A\in \omega_R$.  Let $Y$ be a set so that $Y$ and $\overline{Y}$ are $m$-incomparable with every noncomputable $R$-class. Let $T$ be the least upper bound of $R$ and $E(Y)$. We  will show that the existence of the least upper bound $T$ will imply that there is a reduction $f\colon  R\leqc R$ which avoids the class $A$. By Lemma \ref{avoiding one makes you nca}, this suffices to show that $R$ is \NCA.

  By Lemma \ref{normal form for joins}, we may assume $T=R\oplus E(Y)_{/\sim}$, a pure quotient of $R\oplus E(Y)$.  Since $T\leqc R\oplus E(Y)$, we see that no noncomputable $R$-class $C$ can be collapsed in $T$ to an $E(Y)$-class. This is because then $f\colon T\leqc R\oplus E(Y)$ would give an $m$-reduction from $C\oplus Y$ (or $C\oplus \overline{Y}$) to either some $E(Y)$-class (giving an $m$-reduction of $C$ to $Y$ or $\overline{Y}$) or to an $R$-class (giving an $m$-reduction of $Y$ or $\overline{Y}$ to an $R$-class).
  So, we know $T=R\oplus E(Y)_{/\sim}$ where $\sim$ collapses at most 2 $R$-classes, each of which must be computable, with the odd classes.
  
  Fix any $R$-class $B\neq A$ and let 
\[  
  S:=R\oplus E(Y)_{/(2\min A, 2\min Y +1), (2\min B,2\min \overline{Y}+1)},
\]  
  i.e., we collapse $A$ with the $Y$-class in $E(Y)$ and $B$ with the $\overline{Y}$ class in $E(Y)$.
  Next, consider the reduction $g\colon T\leqc S$. Consider the two $T$-classes of $Y$ and $\overline{Y}$ (possibly collapsed also with computable $R$-classes). Since these do not $m$-reduce to any $R$-class, their $g$-images must intersect the odds. Thus, the image of the evens under $g$, with the exception of two classes, must avoid each class containing the odds. In other words, we have a reduction $h\colon  R\restriction Z\leq R$ where $Z=\overline{C}$ for $C$ the union of the (at most 2) computable $R$-classes which are $\sim$-collapsed in $T$ with odd classes, and $h$ avoids the classes $A$ and $B$. Thus, by extending $h$ to the computable classes, we get a reduction $\hat{h}\colon R\leqc R$ and if $A$ has an $\hat{h}$-preimage, this preimage must be a computable class. If $A$ is not in the image of $\hat{h}$ (e.g., if $T=R\oplus E(Y)$ and $\sim$ does not collapse any computable $R$-class to an $E(Y)$-class), then we are done. So, suppose the class $D$ is computable and is sent to $A$ via $\hat{h}$. Then we apply Corollary~\ref{non-comp classes with comp preimages are avoidable} to show that there is a reduction of $R$ to itself which avoids $A$. 
\end{proof}

%\subsection{Defining $\bLight$ and $\bDark$}\label{sec:defining Id}
We turn to showing that $\bId$ is definable in $\ER$ as the unique \NCA degree minimal over $\bF$. From there, we define $\bLight$ and $\bDark$.

\begin{thm}\label{definability of Id}\label{Id joins with finites}
	In $\ER$, $\bId$ is definable as the unique \NCA degree which is minimal over $\bF$.
\end{thm}
\begin{proof}
	 The fact that $\Id$ is minimal over $\mathcal{F}$ is easy ($\Id\restriction W \equivc \Id_{\abs{W}}$ for any c.e.\ $W$), and $\Id$ is obviously \NCA.
	
%	\begin{lem}
%		$\Id$ has a join with any $R\in \F$.
%	\end{lem}
%	\begin{proof}
%		Let $R\in \F$ be given. We want to show that $R\oplus \Id$ is the least upper bound of $R$ and $\Id$. Let $X$ be any \er above both $R$ and $\Id$. Let $f:R\leqc X$. Let $g\colon \Id\leqc X$ be so that $\range(g)$ is disjoint from the finitely many classes in $\range(f)$. Then $f\oplus g$ is a reduction of $R\oplus \Id$ to $X$. 
%	\end{proof}

	We now verify that $\mathbf{Id}$ is the only minimal \NCA degree. Every other degree minimal over $\bF$ is self-full by Lemma \ref{dark implies self-full} and has a noncomputable class by Lemma \ref{dark minimal implies r.i.}. Clearly any self-full \er with a noncomputable class is not \NCA.
\end{proof}

\begin{corollary}\label{cor:light and dark are definable}
	$\bLight$ and $\bDark$ are definable in $\ER$.
\end{corollary}
\begin{proof}
	$\mathbf{d}\in \bLight$ if and only if $\bId\leq \mathbf{d}$. $\mathbf{d}\in \bDark$ if and only if $\mathbf{d}\notin \bF\cup \bLight$.
\end{proof}

Having defined the degree $\bId$, we wonder which other degrees are definable in $\ER$. In particular, we ask if the degree of the universal ceer is definable:

\begin{question}
	Is the degree of the universal ceer, or equivalently the substructure $\Ceers$, definable in $\ER$?
\end{question}

\section{Covers and Branching}\label{sec:covers}

We now turn our attention to further structural properties in $\ER$. We consider the existence of minimal covers and strong minimal covers, and we explore which degrees are branching. Here, many of the results differ from their analogues in the theory of ceers.

In a degree structure, a \emph{minimal cover} for a degree $\mathbf{d}$ is a minimal upper bound of  $\set{\mathbf{d}}$, i.e., a degree $\mathbf{c}>\mathbf{d}$ such that there is no degree strictly between $\mathbf{c}$ and $\mathbf{d}$; a minimal cover $\mathbf{c}$ of $\mathbf{d}$ is \emph{strong} if anything strictly below $\mathbf{c}$ is bounded by $\mathbf{d}$, i.e., 
\[
(\forall \mathbf{b})(\mathbf{b}<\mathbf{c}\Rightarrow \mathbf{b}\leq \mathbf{d}).
\]
A degree is \emph{branching} if it is the meet of two incomparable degrees.

In $\Ceers$, not all degrees are branching. Andrews and Sorbi~\cite{andrews2019joins} proved that a ceer $R$ is self-full if and only if $R\oplus \Id_1$ is the unique strong minimal cover of $R$. Further, it has the following upward covering property: If $X>R$, then $X\geq R\oplus \Id_1$. This implies that the degree of $R$ cannot branch. In fact, they show that the branching degrees in $\Ceers$ are precisely the non-self-full degrees \cite[Theorem 7.8]{andrews2019joins}. In $\mathbf{ER}$, the situation is quite different. In this section, we will show that every degree has continuum many strong minimal covers, and therefore every degree is branching. 
Before proving these results, we will concentrate on the $\oplus \Id_k$ operation for self-full equivalence relations (where $R\oplus \Id_k>_c R$). We show that though $R\oplus \Id_1$ is a minimal cover of any self-full equivalence relation $R$ (Corollary \ref{cor:minimal covers of self-full}), it is not always a strong minimal cover. That is,  surprisingly and in constract with the case of ceers, there are equivalence relations $R$ such that $R\oplus \Id_1>_c S$, for some $S$,  but $S$ is not computably reducible to $R$.

%\subsection{Covers of self-full degrees}

\begin{thm}
	If $R$ is self-full and $R\leqc S\leqc R\oplus \Id_k$, then there is some $j\leq k$ so that $S\equivc R\oplus \Id_j$. 
\end{thm}
\begin{proof}
	We prove this by induction on $k$. For $k=0$, the result is trivial.
Next, let $f\colon R\leqc S$, $g\colon S\leqc R\oplus \Id_k$, and  suppose that $S$ is not equivalent to $R\oplus \Id_j$ for any $j\leq k$.

	\begin{claim}
		The range of $f$ intersects every $S$-class.
	\end{claim}
	\begin{proof}
		If the range of $f$ did not intersect every $S$-class, then we would have $R\oplus \Id_1\leqc S$. But then we could use the inductive hypothesis, since $R\oplus \Id_1\leqc S \leqc R\oplus \Id_1\oplus \Id_{k-1}$. Thus, we would know that $S\equivc R\oplus \Id_1\oplus \Id_j$ for some $j\leq k-1$, but then it would follow that $S\equivc R\oplus \Id_{j'}$ for some $j'\leq k$.
	\end{proof}

	\begin{claim}
		The range of $g$ intersects every $R\oplus \Id_k$-class. 
	\end{claim}
	\begin{proof}
		If the range of $g$ did not intersect every $R\oplus \Id_k$-class, then we would have $S\leqc R\oplus \Id_{k-1}$. But then, since $R\leqc S\leqc R\oplus \Id_{k-1}$, we could use the inductive hypothesis to show that $S\equivc R\oplus \Id_j$ for some  $j\leq k-1$.
	\end{proof}

	Let $h\colon =g\circ f$ be the composite reduction of $R$ to $R\oplus \Id_k$ through $S$. Fix any odd number $a$ and  let $C_i:=\{x : h\circ (\frac{h}{2})^{(i)}(x) \rel{R\oplus \Id_k} a \}$. Note that  the $C_i$'s so defined for $i\geq 1$ are a uniform sequence of computable $R$-classes. Thus Lemma \ref{avoid computable classes from a uniformly computable sequence of computable classes} yields a contradiction by showing that $R$ is not self-full.
\end{proof}

Applying this to $k=1$, we get that if $R$ is self-full, then $R\oplus \Id_1$ is a minimal cover of $R$.

\begin{corollary}\label{cor:minimal covers of self-full}
Let $R$ be self-full. Then $R\oplus \Id_1$ is a minimal cover of $R$.
\end{corollary}

Now, we will show that, contrary to the case of ceers, there are self-full equivalence relations $R$ so that $R \oplus \Id_1$ is not a strong minimal cover of $R$. To do so, we introduce generic covers of equivalence relations. Intuitively, a generic cover $S$ of a given equivalence relation $R$ codes $R$ into the evens and is generic given this property.

\begin{definition}
A \emph{generic cover} $S$ of an \er $R$ is any \er of the form $R\oplus \Id_{/ \graph(f)}$, where $f\colon \odds \rightarrow \evens$ 
%$f\in \evens^{\odds}$ 
is $1$-generic over the Turing degree of $R$.
\end{definition}

Clearly, $R$ is computably reducible to any
generic cover of $R$ via the map $x\to 2x$. We now see how reductions into the odds must intersect the classes of $S$.

\begin{lemma}\label{lem:generic-dark}
Let $S$ be a generic cover of $R$ and $Z\subseteq \odds$ be an infinite set which is c.e.\ in  the Turing degree of $R$. Then, $Z$ intersects every $S$-class infinitely. It follows that $S\not\leqc R$.
%
%In particular, $Z$ is not a transversal and it is not contained in any one class. 
\end{lemma}

\begin{proof}
Assume that $S$, $R$, and $Z$ are as in the statement of the lemma. In particular, $S=R\oplus \Id_{/\graph(f)}$. Observe that the following sets of strings are c.e.\ in $\deg_T(R)$,
%Take 3:
\[
V_{a,k}:=\set{\sigma\in \evens^{<\odds} : (\exists^k x)(x\in Z\wedge \sigma(x)=2a)}
\]
Further, since $Z$ is infinite, $V_{a,k}$
is dense in $\evens^{<\odds}$. Therefore $f$ meets every $V_{a,k}$ by genericity of $f$, and $Z$ intersects the $S$-class of every even number, so every $S$ class, infinitely often.

Next, suppose $f\colon S\leqc R$ and take any odd number $a$. Let $Z=\{b\in \odds : f(b)\rel{R} f(a)\}$. Necessarily $Z$ is an infinite $R$-c.e.\ set since $Z$ contains $[a]_S\cap \odds$ (and the set $\odds$ intersects every $S$-class infinitely by the above). Therefore, $Z$ meets every $S$-class, contradicting that $f$ is a reduction.
\end{proof}

%Equivalently, if $S$ is a generic cover of $R$, then $S|\odds$ is $R$-dark.
%A notable feature of generic covers is that 

So, $R$ is properly reducible to a generic cover of $R$, but the way in which $S$ covers $R$ is quite different from the way in which $R\oplus \Id_1$ covers $R$:

\begin{lem}\label{generic covers}
	If $S$ is a generic cover of $R$, then, for all $n$, the only \ers which reduce to both $R\oplus \Id_n$ and $S$ are the \ers reducible to $R$.
\end{lem}
\begin{proof}
	Suppose that, for some equivalence relation $X$, there are $f\colon  X\leqc R\oplus \Id_n$ and $g \colon  X\leqc S$. 
	Let $A$ and $B$ be any two $X$-classes. Note that $A,B \leq_m R\oplus \Id_n \equiv_T R$ by Lemma \ref{f respects m-degrees}. Consider the $R$-c.e.\ sets $\odds \cap \ran(g\restriction \overline{A})$ and $\odds \cap \ran(g\restriction \overline{B})$. These must both be finite, as otherwise Lemma \ref{lem:generic-dark} would show that $g\restriction \overline{A}$ would hit $g^\star(A)$ or $g\restriction \overline{B}$ would hit $g^\star(B)$. Thus $\ran(g)\cap \odds$ is finite. So, Lemma \ref{lem:reducing to total pure quotients and finitely in odds} shows that $X\leqc R$.
\end{proof}

In \textbf{Ceers}, $R\oplus \Id_1$ is a strong minimal cover (in fact, the only one) of a given self-full ceer $R$. Hence, any ceer which is below $R\oplus \Id_1$ is already reducible to $R$. But the dual property also holds: $R\oplus \Id_1$ reduces to any ceer which is above $R$ (see \cite[Lemma 4.5]{andrews2019joins} for details). The next theorem uses generic covers to show that these properties both fail in $\ER$.

\begin{thm}\label{thm:plusId1NotAStrongCover}
The following hold.
\begin{enumerate}
\item Let $R$ be any self-full \er. There is $S$ such that $R<_c S$ but $R\oplus \Id_1\not\leqc S$. 
\item There exist a self-full \er $R$ such that, for some $S$, $S<_c R\oplus \Id_1$ but $S\not\leqc R$.
\end{enumerate}
\end{thm}

\begin{proof}
$(1)$:  Let $S$ be a generic cover of $R$. $S$ is above $R$ and, by Lemma \ref{generic covers}, we have that $S$ is incomparable with $R\oplus \Id_1$.

$(2)$: Let $S_0$ be any self-full \er, let $R$ be a generic cover of $S_0$, and denote $S_0\oplus \Id_1$ by $S$. 
	It is immediate that $S\leqc R\oplus \Id_1$ as $S_0\leqc R$. But $S$ and $R$  are incomparable by Lemma \ref{generic covers}. 
\end{proof}

Having shown that $R\oplus \Id_1$ is not a strong minimal cover for some self-full $R$, it is natural to ask whether every self-full degree has a strong minimal cover. The next theorem answers this question affirmatively. In fact, \emph{all} equivalence relations aside from $\Id_1$ have continuum many strong minimal covers, and such covers can be chosen to be self-full.

\begin{thm}\label{continuum many smcovers}
	Let $R$ be any \er $\neq \Id_1$. Then there are continuum many strong minimal covers of $R$ which are self-full. 
\end{thm}
\begin{proof}
	%FROM URI: This argument was definitely broken. I rewrote it after looking back at the old document. I'm not sure if this version or the one in the old document is better.
	In \cite[Theorem~4.10]{andrews2019joins},  it is proven that there is a ceer $E_0$ which satisfies the following properties:
\begin{enumerate}
\item $E_0\restriction \evens = \Id$;
\item  there are infinitely many classes which contain no even number;
\item  if $W$ is any c.e.\ set which intersects infinitely many $E_0$-classes which contain no even number, then $W$ intersects every $E_0$-class.
\end{enumerate}	
There, it is shown that such a  ceer is a self-full strong minimal cover of $\Id$. Here, we let $S_0$ be the quotient of $E_0$ formed by collapsing $2n$ with $2m$ if and only if $n \rel{R} m$. Note that $S_0\restriction \evens = R$.
	
	Let $\mathcal{S}$ be the set of quotients of $S_0$ which collapse every $S_0$-class which contains no even number to exactly one $S_0$-class which does contain an even number. That is, 
	\[
	\mathcal{S}:=\{{S_0}_{/A} : {S_0}_{/A} \restriction \evens = R\text{ and } [\evens]_{{S_0}_{/A}}=\omega\}.
	\]

Since $E_0$, and thus also $S_0$, has infinitely many classes which contain no even number, and $\abs{\omega_R}>1$, we have $\abs{\mathcal{S}}=2^{\aleph_0}$. Thus, there are continuum many elements of $\mathcal{S}$ which are not $\leqc R$, and there is a continuum sized $\leqc$-antichain in $\mathcal{S}$. It suffices to show that for $S\in \mathcal{S}$, if $X<_c S$, then $X\leqc R$.
	It suffices by Remark \ref{rmk:restrictions} to prove that either 
	$S \leqc S\restriction W$ or $S\restriction W \leqc R$ for any c.e.\ set $W$. 
	
	We argue by cases:
	\begin{enumerate}
		\item If $W$ intersects only finitely many $E_0$-classes which do not contain an even number, then we build a reduction of $S\restriction W$ to $R$ as follows:

		Let $a_1,\ldots a_n$ represent the $E_0$-classes which contain no even number and are intersected by $W$. Let $b_1,\ldots b_n$ be even numbers so that $a_i \rel{S} b_i$. Then define $g(x)$ to be the first member of $\evens \cup \{a_i : i\leq n\}$ found to be $E_0$-equivalent to $x$ (note that we are using that $E_0$ is a ceer). Then let $h(x)= g(x)$ if $g(x)$ is even and $h(x)=b_i$ if $g(x)=a_i$. This gives a reduction of $S\restriction W$ to $S$ whose range is contained in the evens. So, this gives a reduction of $S\restriction W$ to $S\restriction \evens = R$.
		
		\item If $W$ intersects infinitely many $E_0$-classes which do not contain an even number, then we know that $W$ intersects every $E_0$-class. We then give a reduction of $S$ to $S\restriction W$ by sending $x$ to the first member of $W$ found to be $E_0$-equivalent to $x$. Since $S$ is a quotient of $E_0$, this is the identity map on classes, so a reduction of $S$ to $S\restriction W$.
	\end{enumerate}
%	
%	\begin{enumerate}
%	\item On the one hand, suppose that $W$ hits only finitely many $E_0$-classes with representatives $a_0,\ldots,a_k$. For each $a_i$, fix non-uniformly an even number $b_i$ such that $a_i S b_i$. Then consider the following computable function $f:W\to \omega$,
%	\[
%	f(x)=\begin{cases}
%	b_i &\text{$x$ is odd and $xE_0 a_i$,}\\
%	x &\text{otherwise}.
%	\end{cases}
%	\]
%	Since any odd number in $W$ is equivalent to exactly one of the $a_i$'s and $E_0$ is a ceer, $f$ is well-defined and computable. Observe that $f$ reduces $S\restriction W$ to $S$ and $\range(f)\subseteq \evens$. Since $S\restriction \evens$ is equivalent to $R$, this means that $f$ gives the desired reduction from $S\restriction W$ to $R$.
%	\item On the other hand, suppose that $W$ hits infinitely many $E_0$-classes. Then, by the property of $E_0$, it hits every $E_0$-class. We obtain $S\leqc S\restriction W$ via the map $x\mapsto  y$, where $y$ is the first seen in $W$ so that $x \rel{E_0} y$: this is a reduction of $S$ to $S\restriction W$, since $x$ is always sent to a member of its own class.
%	\end{enumerate}
%	
%	Combining $(1)$ and $(2)$, we conclude that $S$ is a strong minimal cover of $R$.
	Lastly, we check that $S$ is self-full. Suppose $f$ is a function reducing $S$ to itself. Let $W$ be $\range(f)$. Since $R<_c S$, we cannot be in case (1) above, so $W$ must intersect every $E_0$-class, so also every $S$-class.
\end{proof}

\begin{cory}
	In $\ER$, every degree is branching.
\end{cory}
\begin{proof}
	Every degree $\mathbf{d}$ has two incomparable strong minimal covers. The meet of these two degrees is $\mathbf{d}$.
%	If $R$ is non-self full, then there are incomparable strong minimal covers. These must meet to $R$.
%	
%	If $R$ is self-full, then let $S_1$ be $R\oplus \Id_1$. Let $S_2$ be a generic cover of $R$. It follows from Lemma \ref{generic covers} that $S_1$ and $S_2$ meet to the degree of $R$.	
\end{proof}

So, contrary to the case of ceers, the self-full equivalence relations cannot be characterized in terms of their strong minimal covers. We ask:

\begin{question}
Is the collection of self-full degrees first-order definable in $\ER$?
\end{question}

\section{The complexity of the first-order theory of $\ER$}\label{sec:arithmetic}
In this last section, we characterize the complexity of $\Th(\ER)$,  the first-order theory of $\ER$. Our analysis contributes to a longstanding research thread. Indeed, computability theorists have been investigating the first-order complexity of degree structures generated by reducibilities for decades. 

Since a reducibility $r$ is typically a binary relation on subsets of $\omega$,   one can effectively
translate first-order sentences regarding the corresponding degree structure $\mathcal{D}_r$  to second-order
sentences of arithmetic, obtaining a $1$-reduction from 
$\Th(\mathcal{D}_r)$ to $\Th^2(\mathbb{N})$. Remarkably, the converse reduction  often holds, e.g., the first-order theories of the following degree structures are $1$-equivalent (and so, by Myhill Isomorphism Theorem, computably isomorphic) to
second-order arithmetic: the Turing degrees $\mathcal{D}_T$~\cite{simpson1977first}; the $m$-degrees $\mathcal{D}_m$, the $1$-degrees $\mathcal{D}_1$, the $tt$-degrees $\mathcal{D}_{tt}$, the $wtt$-degrees $\mathcal{D}_{wtt}$~\cite{nerode1980second}; and the enumeration degrees $\mathcal{D}_{e}$~\cite{slaman1997definability}.  
Here, we add $\ER$ to this list, namely, we prove:

\begin{thm}\label{ER equivalent to 2order arithmetic}
$\Th(\ER)$ is computably isomorphic to $\Th^2(\mathbb{N})$.
\end{thm}
In fact, we will show that the theorem is also true for each of the definable substructures $\bDark$ and $\bLight$ of $\ER$.

\subsection{Our strategy}\label{coding graphs}
Equivalence relations are  straightforwardly encoded into subsets of $\omega$, hence $\Th(\ER)\leq_1 \Th^2(\mathbb{N})$ trivially holds. So, to prove Theorem \ref{ER equivalent to 2order arithmetic}, it suffices to prove the converse reduction.  Our strategy for coding  second-order arithmetic into $\ER$ is based on coding  all countable graphs as second order objects into this degree structure. The justification for such approach relies on well-known facts. Second-order arithmetic is $1$-reducible to second-order logic on countable sets, which is in turn $1$-reducible to the theory of second order countable graphs~\cite{lavrov1962effective}. So, one can effectively translate any question about second-order arithmetic into a question about a graph which encodes the standard model of Robinson's arithmetic Q. 

Finally, let us mention that our encodings are similar to the way in which graphs  are coded in $\mathbf{Ceers}$, as in \cite{andrews2020theory}. But there are three major differences. Firstly, in what follows we code any countable graph, rather than just computable graphs. Secondly, we must code subsets of the set of vertices of our graph. Thirdly, since we are giving codes for subsets, we do not need to code functions between different codings of natural numbers; that means that we do not need to distinguish the natural numbers from non-standard models of Robinson's Q as being embeddable into any other such model (thus needing to code functions), because the second order theory distinguishes the standard model of Robinson's Q as the only one with no proper inductive subset.

\subsection{Coding graphs into Dark}

To code graphs in $\mathbf{Dark}$, we heavily use dark minimal degrees:  We fix  a collection $\{D_i: i\in\omega\}$ of pairwise nonequivalent dark minimal equivalence relations. In fact, since $\mathbf{Ceers}$ is an initial segment of $\ER$, we may choose dark minimal ceers (as constructed in \cite{andrews2019joins}). 

%Why recall this now?
%Recall that  the equivalence classes of any dark minimal equivalence relation are pairwise computably inseparable, as shown by Lemma \ref{dark minimal implies r.i.}.

\begin{defn}
	Let $\mathbf{d_1},\mathbf{d_2}$ be two dark minimal degrees. We say that incomparable degrees $\mathbf{a},\mathbf{b}$ are a \emph{covering pair} of $\mathbf{d_1},\mathbf{d_2}$ if, for each $\mathbf{x}\in \{\mathbf{a},\mathbf{b}\}$, the set of dark minimal degrees below $\mathbf{x}$ is precisely $\{\mathbf{d_1},\mathbf{d_2}\}$, and there is no $\mathbf{y}< \mathbf{a},\mathbf{b}$ so that $\mathbf{d_1},\mathbf{d_2}< \mathbf{y}$.
\end{defn}

%LUCA: If the above is wrong, change it back to:
%\begin{defn}
%	Let $\mathbf{d_1},\mathbf{d_2}$ be two dark minimal degrees. We say that degrees $\mathbf{a},\mathbf{b}$ are a \emph{covering pair} of $\mathbf{d_1},\mathbf{d_2}$ if, for each $\mathbf{x}\in \{\mathbf{a},\mathbf{b}\}$, the set of dark minimal degrees below $\mathbf{x}$ is precisely $\{\mathbf{d_1},\mathbf{d_2}\}$, and there is no $\mathbf{y}\leq \mathbf{a},\mathbf{b}$ so that $\mathbf{d_1},\mathbf{d_2}\leq \mathbf{y}$.
%\end{defn}

We now describe how to encode a countable graph by parameters in $\mathbf{Dark}$.

\begin{defn}\label{definition of $G_c$}
	For any degree $\mathbf{c}$, let $G_{\mathbf{c}}$ be the graph with vertex set composed of the dark minimal degrees below $\mathbf{c}$ and  edges the collection of pairs $\mathbf{d_1},\mathbf{d_2}$ so that there are distinct $\mathbf{a},\mathbf{b}\leq \mathbf{c}$ which form a covering pair of $\mathbf{d_1},\mathbf{d_2}$.
\end{defn}

The next lemma provides an easy way of forming covering pairs of dark minimal equivalence relations.

\begin{lem}\label{covering pairs exist}
	If $D,E$ are dark minimal \ers, then 
	$D\oplus E$ and $D\oplus E_{/(0,1)}$ form a covering pair of $D$ and $E$.
\end{lem}
\begin{proof}

	It is immediate that $D$ and $E$ are both computably reducible to $D\oplus E$ and $D\oplus E_{/(0,1)}$ (the latter being a pure quotient). 
	
	We show that the only dark minimal degrees below either $D\oplus E$ or $D\oplus E_{/(0,1)}$ are the degrees of $D$ and $E$.
	
	Suppose $f\colon  X\leqc D\oplus E$, for a dark minimal $X$. Since $X$ is dark minimal, its equivalence classes are computably inseparable by Lemma \ref{dark minimal implies r.i.}, so $\range(f)$ must be either contained in the evens or the odds, which implies $X\leqc D$ or $X\leqc E$. But then $X\equivc D$ or $X\equivc E$, by minimality of $D$ and $E$.
	
	On the other hand, suppose $f\colon X\leqc D\oplus E_{/(0,1)}$, for a dark minimal  $X$. Since the equivalence classes of $X$ are computably inseparable by Lemma \ref{dark minimal implies r.i.}, $\range(f)$ is contained in either
	\begin{enumerate}
		\item  $\evens \cup [1]_{D\oplus E_{/(0,1)}}$;
		\item  or $\odds\cup [0]_{D\oplus E_{/(0,1)}}$.
	\end{enumerate}	
	Without loss of generality, we assume the former. Let $h$ be the function given by $h(x)=x$ if $x$ is even and $0$ if $x$ is odd. Then $h\circ f\colon  X\leqc D\oplus E_{/(0,1)}$ and $\range(h\circ f)\subseteq \evens$. This induces a  reduction of $X$ to  $D$. But then $X\equivc D$, by minimality of $D$.

	Next, we consider the degrees strictly below $D\oplus E$ which might bound both $D$ and $E$. Suppose that $X\leqc D\oplus E$. Then by Lemma \ref{lem:below a join}, $X\equivc D_0\oplus E_0$ where $D_0\leqc D$ and $E_0\leqc E$. So either
	\begin{enumerate}
		\item $X\in \F$,
		\item or $X\equivc D\oplus E$,
		\item or $X\equivc D\oplus F$ for some $F\in \F$,
		\item or $X\equivc E\oplus F$ for some $F\in \F$.
	\end{enumerate}	
	In the first case, $X$ obviously does not bound $D$ or $E$. In the second, $X$ is not strictly below  $D\oplus E$. In cases $(3)$ and $(4)$, $X$ does not bound both $D$ and $E$. To see this, suppose $X\equivc D\oplus F$ for some $F\in \F$. Then any reduction of $E$ to $X$ gives a reduction of $E$ to $D\oplus F$. But by computable inseparability of the classes of $E$, this reduction is either contained in the evens, giving $E\leqc D$, or contained in the odds,  giving $E$ is finite, either way leading to a contradiction. Thus, there is no \er $X$ which is strictly reducible to $D\oplus E$ and bounds both $D$ and $E$. 
	
	Next we observe that $D\oplus E$ and $D\oplus E_{/(0,1)}$ are incomparable. The fact that $D\oplus E\not\leqc D\oplus E_{/(0,1)}$ follows from darkness of $D\oplus E$ and Lemma \ref{dark not reduces to proper quotient}. The fact that $D\oplus E_{/(0,1)}\not<_c D\oplus E$ follows from the previous paragraph.
%	
%	
%	
%	
%	
%	
%
% 
%
% \lucainsert{On the other hand, towards a contradiction suppose that there is $i\colon D\oplus E_{/(0,1)}  \leqc D\oplus E$. It is convenient to denote by $Z$ the equivalence class of $D\oplus E_{/(0,1)}$ containing $\{0,1\}$, i.e.,
%	\[	
%	Z:={[0]}_{{D\oplus E}_{/(0,1)}}.
%	\]	
%	Without loss of generality, assume that $i^\star$ maps $Z$ to a $D$-class, i.e., $i(Z)\subseteq \evens$. Next, define the following computable set:
%	\[
%	C:=\{x: i(2x+1) \mbox{ is even}\}.
%	\]	 
%	Since $[1]_{E}\subseteq Z$, we deduce  that $[1]_{E}\subseteq C$. Moreover, there must be $V\in\omega_E$  such that $V\cap C=\emptyset$, as otherwise $E\leqc D$. But then $C$ computably separates $[1]_E$ from $V$, contradicting Lemma~\ref{dark minimal implies r.i.}.
%}

Finally, by incomparability of $D\oplus E$ and $D\oplus E_{/(0,1)}$, any degree below both would have to be strictly below $D\oplus E$, so cannot bound both $D$ and $E$.
\end{proof}

We are ready to show that we can uniformly code any countable graph as a second order structure into $\bDark$, which, combined with the remarks offered in Section \ref{coding graphs},  will yield the following theorem.

\begin{thm}
	The theory of the degree structure $\bDark$ is computably isomorphic to second-order arithmetic.
\end{thm}
\begin{proof}
We first embed any countable graph as a first-order structure into $\bDark$. 
	\begin{lem}\label{build C}
		For any countable graph $G$, there is some $\mathbf{c}\in \bDark$ so ${G_{\mathbf{c}}\cong G}$. 
	\end{lem}
	\begin{proof}
		We may assume that the universe of $G$ is $\omega$ (if $G$ is finite, then the dark ceer $C$ constructed below can be taken simply as the uniform join of $D_i$ and $D_u\oplus {D_v}_{/(0,1)}$ for pairs where $u \rel{G} v$). 
		Recall that $\{D_i: i\in \omega\}$ represents a collection of distinct dark minimal degrees. 

%FROM URI: Trying to implement Luca's idea about immunity:

Let $X$ be the collection of \ers 
\[
\{D_i : i\in \omega\}\cup \{{D_i\oplus D_j}_{/(0,1)} : i \rel{G} j\}
\]
 and fix an enumeration of $X=(X_i)_{i\in \omega}$.
Fix $S$ to be an immune set. Then we define $C$ by $\langle x,i\rangle \rel{C} \langle y,j\rangle $ if and only either $i=j$ is the $n$th element of $S$ and $x \rel{X_n} y$ or $i,j\notin S$.	

We now argue that $C$ is dark and $G_{\mathbf{c}}\cong G$, where $\mathbf{c}$ is the degree of $C$. The proof is split into several claims.

\begin{claim}\label{lem:coding element is dark}
$C$ is dark.
\end{claim}
\begin{proof}
If $W_e$ intersects infinitely many columns of $\omega$, then by immunity of $S$, it enumerates two elements $\langle x, i\rangle, \langle y, j\rangle$ with $i,j\notin S$. But then $\langle x, i\rangle \rel{C} \langle y, j\rangle$ and $W_e$ is not a transversal.

If $W_e$ intersects only finitely many columns, then $W_e$ is enumerating a subset of $Y=\{\langle x, i\rangle : i\leq m\}$ for some $m$. But $C\restriction Y$ is equivalent to a finite uniform join of dark ceers $X_i$. Thus $W_e$ cannot be a transversal.
%
%Suppose $W_e$ is infinite. If $W_e$ intersects some column infinitely often, then the darkness of each column ensures that $W_e$ is not a transversal. Otherwise, $W_e$ intersects infinitely many columns, so by item $(b)$ of stage $3e$ of the construction, we have that there exists  $i\in W_e\cap [0]_C$. Then $W_e\setminus \{i\}$ is also an c.e.\ set which intersects infinitely many columns. Thus, it must also intersect $[0]_C$. But then $W_e$ is not a transversal. 
\end{proof}

Next we see that the only dark minimal degrees bounded by $\mathbf{c}$, i.e., those which are vertices in $G_\mathbf{c}$, are $\{D_i : i\in \omega \}$.

	\begin{claim}\label{No Rando Below C}
	If $D\leqc C$ and $D$ is dark minimal, then $D\equivc D_u$ for some $u$.	
\end{claim}
\begin{proof}
	Since $D$ is dark minimal, its classes are computably inseparable by Lemma~\ref{dark minimal implies r.i.}. So, either $D\leqc D_u$, for some $u$, or $D\leqc D_i\oplus {D_j}_{/(0,1)}$, for some pair $i,j$. In the former case, dark minimality of $D_u$ ensures $D\equivc D_u$, and in the latter case Lemma \ref{covering pairs exist} ensures $D\equivc D_i$ or $D\equivc D_j$.
\end{proof}

		We now know that the map $i\mapsto \mathbf{d}_i$ is onto $G_{\mathbf{c}}$. It only remains to show that it is an embedding of $G$.

\begin{claim}
If $u \rel{G} v$, then $u \rel{G_\mathbf{c}} v$.
\end{claim}

\begin{proof}
	There are three columns of $C$, coding $D_u, D_v$, and ${D_u\oplus D_v}_{/(0,1)}$. Therefore, $D_u\oplus D_v, {D_u\oplus D_v}_{/(0,1)}$ are both $\leqc C$. By Lemma \ref{covering pairs exist}, these form a covering pair of $D_u$ and $D_v$, so we have $u \rel{G_\mathbf{c}} v$.
%First, observe that $C$ codes both $D_u$ and $D_v$, i.e., there are numbers $\hat{u}$ and $\hat{v}$ (appointed, respectively, at stages $u$ and $v$) such that $\col_{\hat{u}}$ and $\col_{\hat{v}}$ are coding columns for $D_u$ and $D_v$. Such columns are never collapsed in further stages. It follows that $D_u$ and $D_v$, which are dark and minimal, are reducible to $C$, and so is their uniform join $D_u\oplus D_v$. Moreover, since $u\rel{G}v$ holds, there must a stage of the form $3e+2$ in which a column codes $D_u\oplus {D_v}_{/(0,1)}$, and therefore such equivalence relation is also below $C$. By Lemma \ref{covering pairs exist}, $D_u\oplus D_v$ and $D_u\oplus {D_v}_{/(0,1)}$ form a covering pair of $D_u$ and $D_v$, and therefore, by Definition \ref{definition of $G_c$}, we have that $u \rel{G_\mathbf{c}} v$.
\end{proof}

\begin{claim}
			If $u \rel{G_\mathbf{c}} v $, then $u \rel{G} v$. 
		\end{claim}
		\begin{proof}
			Suppose that $\mathbf{a},\mathbf{b}\leq \mathbf{c}$ form a covering pair of $\mathbf{d}_u$ and $\mathbf{d}_v$ and $u,v$ are not adjacent in $G$.
			Let $A\in \mathbf{a}$, $B\in \mathbf{b}$, $D_u\in \mathbf{d}_u$ and $D_v\in \mathbf{d}_v$. Consider the composite reductions $f_u\colon  D_u\leqc A\leqc C$ and $f_v\colon  D_v\leqc A\leqc C$. By computable inseparability of the classes of $D_u$ (Lemma \ref{dark minimal implies r.i.}), $\range(f_u)$ must be contained in a single column of $C$. By incomparability of the dark minimal \ers and Lemma \ref{covering pairs exist}, this column must be either $D_u$ or $D_u\oplus {D_w}_{/(0,1)}$ for some $w$ with $u \rel{G} w$. In particular, the column used for $f_u$ cannot be the same as the column used for $f_v$. It follows that $D_u\oplus D_v\leqc A$. Similarly for $B$, contradicting $\mathbf{a}$ and $\mathbf{b}$ forming a covering pair of $\mathbf{d}_u,\mathbf{d}_v$.
		\end{proof}

	This completes the proof of Lemma \ref{build C}.
\end{proof}

	Next, we show that for any $\mathbf{c}$, we can code any subset of $G_\mathbf{c}$.

	\begin{lem} Let $E$ be a countable set of dark minimal degrees.
		There is a degree $\mathbf{a}\in \bDark$ so that the set of dark minimal degrees $\leq \mathbf{a}$ is exactly $E$.
	\end{lem}
	\begin{proof}
Apply the construction of the dark equivalence relation $C$ of Lemma \ref{build C} to the empty graph and the collection of degrees in $E$. That is, let $(E_i)_{i\in \omega}$ be dark minimal \ers representing the classes in $E$. Then let $\langle x, i \rangle \rel{C} \langle y, j \rangle$ if and only if $i=j$ is the $n$th element of $S$ (a fixed immune set) and $x \rel{E_{n}} y$ or if $i,j\notin S$.
Lemma \ref{No Rando Below C} shows that the degrees of dark minimal \ers below $C$ are precisely $E$, and Lemma \ref{lem:coding element is dark} shows that $C$ is dark.
	\end{proof}
	
	For $\mathbf{a}\in \bDark$, let $M_\mathbf{a}$ be the set of dark minimal degrees $\leq \mathbf{a}$.
	Put together, we now know that every second order countable graph is encoded as $(G_{\mathbf{c}}, \mathcal {A})$ for some $\mathbf{c}\in \bDark$, where $\mathcal{A}$ is the set of $M_\mathbf{a}$ for $\mathbf{a}\in \bDark$ which are contained in $G_\mathbf{c}$.
	
	So, $\Th(\bDark)$ is $\geq_1$ the theory of second order countable graphs. As remarked in Section \ref{coding graphs}, this is enough to conclude that $\Th(\bDark)$ is computably isomorphic to second-order arithmetic. Then, Theorem \ref{ER equivalent to 2order arithmetic} immediately follows from the fact that $\bDark$ is definable in $\ER$ (Corollary \ref{cor:light and dark are definable}). 	
\end{proof}

\subsection{Coding graphs into Light}
We now focus on light degrees, with the goal of showing that $\Th(\bLight)$ is also computably isomorphic to second-order arithmetic. The encoding of graphs in the light degrees will be as follows:

\begin{defn}
	A degree $\mathbf{e}$ is a light minimal degree if $\bId<\mathbf{e}$ and there is no $\mathbf{x}$ so that $\mathbf{Id}< \mathbf{x}<\mathbf{e}$.
	
	Let $\mathbf{e}_1,\mathbf{e}_2$ be two light minimal degrees. We say that $\mathbf{a},\mathbf{b}$ are a \emph{light covering pair} of $\mathbf{e}_1,\mathbf{e}_2$ if for each $\mathbf{x}\in \{\mathbf{a},\mathbf{b}\}$, the set of light minimal degrees below $\mathbf{x}$ is precisely $\{\mathbf{e}_1,\mathbf{e}_2\}$ and there is no $\mathbf{y}$ below $\mathbf{a}$ and $\mathbf{b}$ which is above $\mathbf{e}_1,\mathbf{e}_2$.
\end{defn}

\begin{defn}
	For  a light degree $\mathbf{c}$, let $H_{\mathbf{c}}$ be the graph with vertices the light minimal degrees below $\mathbf{c}$ and edges the collection of pairs $\mathbf{e}_1,\mathbf{e}_2$ so that there are $\mathbf{a},\mathbf{b}\leq \mathbf{c}$ which form a light covering pair of $\mathbf{e}_1,\mathbf{e}_2$.
\end{defn}

%\begin{lem}
%	There are infinitely many dark minimal degrees that bound no \er in $\F\setminus \I$.
%\end{lem}
%\begin{proof}
%	Let $X$ be any dark minimal ceer. As the ceers form an initial segment in $\ER$, it follows that $X$ bounds no finite \er outside of $\I$.
%\end{proof}

We now show that we can uniformly encode every second order countable graph into $\bLight$. 

\begin{thm}
	The theory of $\bLight$ is computably isomorphic to second-order arithmetic.
\end{thm}
\begin{proof}
Rather than directly defining light covering pairs of light minimal degrees (as we did in Lemma \ref{covering pairs exist}), we inherit them from the dark case through the following map: let $\iota$ be the map from $\Dark\cup \F$ to $\Light$ given by $\iota (D)=D\oplus \Id$, and $\biota$ the induced map on degrees. The next two claims give two crucial properties of $\iota$.

	\begin{claim}\label{iota is hom}
		$\iota$ gives a homomorphism of $\Dark\cup \F$ into $\Light$ whose image is an initial segment.
	\end{claim}
	\begin{proof}
		It is immediate that $D\leqc E$ implies $\iota(D)\leqc \iota(E)$. Now, suppose $\Id\leqc X\leqc \iota(D)=D\oplus \Id$, for some equivalence relation $X$. From Lemma~\ref{lem:below a join}, it follows $X\equivc D_0\oplus A$, where $D_0\leqc D$ and $A\leqc \Id$. Since $D_0$ is dark or finite, $A$ must be light, since $\Id\leqc X$. So, $A\equivc \Id$. Thus, $X\equivc D_0\oplus \Id=\iota(D_0)$.
	\end{proof}
	
	%FROM URI: I think this is not used below.
%	\begin{claim}\label{iota squishes I}
%		If $\iota(D)\leqc \iota(E)$, then $D\leqc E\oplus \Id_k$ for some $k\in \omega$.
%	\end{claim}
%	\begin{proof}
%		$\iota(D)\leqc \iota(E)$ obviously implies that $D\leqc E\oplus \Id$. Any such reduction can hit only finitely many odd numbers as otherwise $D$ would be light. This gives a reduction to $E\oplus \Id_k$, where $k$ is the size of the range of the reduction in the odd numbers.
%	\end{proof}

\begin{claim}\label{iota of dark minimal ceers are light minimal}
	If $D$ is a dark minimal ceer, then $\iota(D)$ is of light minimal degree.
\end{claim}
\begin{proof}
	Suppose $\Id<_c X\leqc \iota(D)$. Then by the proof of Claim \ref{iota is hom}, $X\equivc \iota(E)$ for some $E\leq D$. But $D$ is a ceer, so $E$ cannot be in $\F$ as that would make $E\in \I$ and $X\equivc \Id$. So, $E\in \Dark$, and thus $E\equivc D$ by dark minimality of $D$.
\end{proof}

Lemma \ref{build C} guarantees that any graph $G$ is encodable into $\bDark$ via some $G_{\mathbf{c}}$. The next lemma says that we can use $\biota$ to transfer our coding of graphs into $\bDark$ into an encoding in $\bLight$.

	\begin{lem}\label{$G_c$ embeds into light}
		For any countable graph $G$, there is a degree $\mathbf{c}\in \Dark$ so that $G_{\mathbf c}\cong G$ is isomorphic to a substructure of $H_{\biota(\mathbf{c})}$
	\end{lem}
	\begin{proof}
		Fix dark minimal ceers $D_i\in \mathbf{d}_i$ and let $\mathbf{c}$ be as constructed in Lemma \ref{build C} so $G_{\mathbf c}\cong G$. 	
%		We first claim that $\iota(\mathbf{c})$ bounds no member of $\F\smallsetminus\I$.
%		
%		\begin{claim}\label{c bounds not finite}
%			 The degree $\iota(\mathbf{c})$ bounds no member of $\F\smallsetminus \I$.
%		\end{claim}
%		\begin{proof}
%			Let $F\in \F$ be reducible to $C\oplus \Id$ via $g$. We want to show $F\equiv \Id_k$, for some $k$. Each column of $C\oplus \Id$ (i.e. column of $C$ or $\Id$ itself) is a ceer, and the range of $F$ can intersect only finitely many columns. Therefore, $F$ reduces to a finite uniform join of ceers, showing that $F$ is a ceer. But every finite ceer is in $\I$.
%%			
%%			For each $i\in U$, let $F_i$ be the finite equivalence relation induced by the pre-images of the $i$th column in $C$. Then $F$ is reducible to $\oplus_{i\in U} F_i \oplus \Id_1$ (where $\Id_1$ is needed to account for the columns not in $S$). Having chosen the $D_i$'s as ceers, each column of $C$ is a ceer, so each $F_i$ is a finite ceer. Thus, $F$ is equivalent to some $\Id_k$.
%		\end{proof}
		Lemma \ref{iota of dark minimal ceers are light minimal} shows that every $\biota(\mathbf{d}_i)$ is in $H_{\biota(\mathbf{c})}$. Let $X$ be the subset of vertices in $H_{\biota(\mathbf{c})}$ comprised of $\biota(\mathbf{d}_i)$ for $i\in \omega$. We do not claim that there are no other light minimal degrees bounded by $\biota(\mathbf{c})$. We now show that $\biota$ gives an isomorphism of $G_{\mathbf{c}}$  with the substructure of $H_{\biota(\mathbf{c})}$ with universe $X$.	
		
By Claim \ref{iota is hom}, $\biota$ gives a homomorphism of the degrees below $\mathbf{c}$ onto the light degrees below $\biota(\mathbf{c})$. We argue that such a homomorphism, when  restricted to the dark minimal degrees and their covering pairs, is in fact an embedding. 

First observe that each distinct pair of dark minimal $D_i$ and $D_j$ below $\mathbf{c}$ are sent via $\iota$ to incomparable degrees. Indeed, if $\iota(D_i)\leqc \iota(D_j)$, then $D_i\leqc D_j\oplus \Id$. By the computable inseparability of the classes of $D_i$, the reduction is either to $D_j$ or $\Id_k$, both of which are impossible.
		
Now, for distinct $D_i,D_j$, observe that $\iota(D_i\oplus D_j)$ and $\iota(D_i\oplus {D_j}_{/(0,1)})$ are sent to incomparable degrees. To see this, recall that, by Lemma \ref{covering pairs exist}, $D_i\oplus D_j$ and $D_i\oplus {D_j}_{/(0,1)}$ are incomparable. Since neither of these have a computable class (because this would contradict the computable inseparability of the equivalence classes of $D_i$ and $D_j$, granted by Lemma \ref{dark minimal implies r.i.}), it follows that neither can reduce to the other $\oplus \Id$, as such a reduction could not make any use of $\Id$.

\begin{claim}	
If $\mathbf{d}_i \rel{G_{\mathbf{c}}} \mathbf{d_j}$, then $\biota(\mathbf{d}_i) \rel{H_{\biota(\mathbf{c})}} \biota(\mathbf{d}_j)$.
\end{claim}	
\begin{proof}	
Let $\mathbf{d}_i \rel{G_{\mathbf{c}}} \mathbf{d}_j$. To show that $\iota(\mathbf{d}_i) \rel{H_{\iota(\mathbf{c})}} \iota(\mathbf{d}_j)$ holds, we need to check that $\iota(D_i\oplus D_j)$ and $\iota(D_i\oplus {D_j}_{/(0,1)})$ form a light covering pair of $\iota(D_i)$ and $\iota(D_j)$. It only remains to check that there is no $Y\leqc \iota(D_i\oplus D_j)$, $\iota(D_i\oplus {D_j}_{/(0,1)})$ such that $\iota(D_i),\iota(D_j)\leqc Y$. Suppose that such a $Y$ existed. Consider the composite reduction $f_i\colon  D_i\leqc Y\leqc D_i\oplus D_j\oplus \Id$. The computable inseparability of the classes of $D_i$ and the incomparability of $D_i$ and $D_j$ force  $f_i$ to go into the first column. Similarly, the reduction of $f_j\colon  D_j\leqc Y\leqc D_i\oplus D_j\oplus \Id$ must go into the second column. It follows that $D_i\oplus D_j\leqc Y$. Since $Y\leqc \iota(D_i\oplus {D_j}_{/(0,1)})$, there is a reduction $D_i\oplus D_j\leqc D_i\oplus {D_j}_{/(0,1)} \oplus \Id$, and thus $\iota(D_i\oplus D_j)\leqc \iota(D_i\oplus {D_j}_{/(0,1)})$, but we have already established that these are incomparable. 
%But, since $D_i$ and $D_j$ have no computable classes, such a reduction would give that $D_i\oplus D_j\leqc D_i\oplus {D_j}_{/(0,1)}$, contradicting Lemma \ref{dark not reduces to proper quotient}. So, no such $Y$ can exist. 
%		
%		
%		We conclude that $\mathbf{d}_i\rel{G_{\mathbf{c}}} \mathbf{d}_j$ implies $\iota(\mathbf{d}_i) \rel{H_{\iota(\mathbf{c})}} \iota(\mathbf{d}_j)$.
\end{proof}

	\begin{claim}	
If $\biota(\mathbf{d}_i) \rel{H_{\biota(\mathbf{c})}} \biota(\mathbf{d_j})$, then $\mathbf{d}_i \rel{G_{\mathbf{c}}} \mathbf{d}_j$.
\end{claim}
\begin{proof}
		Let $\biota(\mathbf{d}_i) \rel{H_{\biota(\mathbf{c})}} \biota(\mathbf{d}_j)$, and let $\iota(A_0)\in \mathbf{a}, \iota(B_0)\in \mathbf{b}$ be a light covering pair of $\biota(\mathbf{d}_i), \biota(\mathbf{d}_j)$. By the computable inseparability of the classes of $D_i$ and $D_j$, $D_i,D_j\leqc A_0$ and $D_i,D_j\leqc B_0$. Since $\biota$ is a homomorphism onto the light degrees below $\biota(\mathbf{c})$, any $\mathbf{y}$ witnessing that $\mathbf{a},\mathbf{b}$ is not a covering pair of $\mathbf{d}_i,\mathbf{d}_j$ would be so that $\biota(\mathbf{y})$ witnesses $\mathbf{a},\mathbf{b}$ are not a light covering pair of $\biota(\mathbf{d}_i)$ and $\biota(\mathbf{d}_j)$. Thus we have $\mathbf{d}_i \rel{G_{\mathbf{c}}} \mathbf{d}_j$.
\end{proof}
This concludes the proof of Lemma \ref{$G_c$ embeds into light}
	\end{proof}

Next, we show that we can code any subset of any countable set of vertices. This will be used both for encoding the second order part of graphs and also for selecting the substructure of $H_{\biota(\mathbf{c})}$ which is isomorphic to $G$.
	
	\begin{lem}\label{coding sets in light}
		Let $\{\mathbf{b}_i : i\in \omega\}$ be a collection of distinct light minimal degrees and $S\subseteq \omega$. Then, there is a degree $\mathbf{c}$ so that $\mathbf{b}_i\leq \mathbf{c}$ if and only if $i\in S$.
	\end{lem}
	\begin{proof}
		Fix a sequence of representatives $L_i\in \mathbf{b}_i$.
		Intuitively, we construct $X\in \mathbf{c}$ to encode each $L_i$ with $i\in S$ on the columns of $\omega$ and then generically collapse equivalence classes between columns.  Enumerate $S=\{a_0<a_1<\ldots \}$.
		
		First we define $X_0$ by
		\[
\langle n, i \rangle \rel{X_0} \langle m, j \rangle \Leftrightarrow i=j \wedge (n \rel{L_{a_i}} m).
\]

Let $\col_i=\set{\langle x, i\rangle : x\in \omega}$, i.e. the $i$th column of $\omega$. 
For all $i$, denote by $T_i\subseteq \col_i$ a transversal of $X_0$ which hits all  classes contained in the $i$th column.   Next, let $(f_i)_{i\in \omega}$ be a (mutually) $1$-generic sequence of permutations of $\omega$ over a Turing degree which computes every $L_i$.

Then let $X={X_0}_{/Z}$ with
\[
Z = \{(T_u[v], T_0[f_u(v)]) : u,v\in \omega\},
\]
 where we let $T_u[v]$ denote the $v$th element of $T_u$ (i.e., $Z$ collapses the $v$th class in the $u$th column to the $f_u(v)$th class in the $0$th column of $X_0$).

\begin{claim}
For all $i\in S$, $L_i\leqc X$.
\end{claim}
\begin{proof}
This follows from the fact that $X_0$ encodes each $L_i$ for $i\in S$ as a column, and the quotient $X$ does not collapse equivalence classes from the same column.
\end{proof}

Suppose towards a contradiction that $g\colon L_j\leq X$ for some $j\notin S$.
\begin{claim}\label{no transversal on two columns}
	There is some $k$ so that $\ran(g)\subseteq^* \col_k$.
%Let $B$ be a $W$-c.e.\ transversal of $X$. Then $W\subseteq^* \col_j$ for some $j$.
%
%If $W$ computes a transversal  $B$ of $X$, then  there cannot be  two distinct transversals $T_i$ and $T_j$ so that $|B\cap T_i|=\infty$ and $|B\cap T_j|=\infty$.
\end{claim}

\begin{proof}
	Let $V$ be the set of finite sequences of finite injective partial maps $(p_i)_{i\leq m}$ so that for some $x,y$, letting $i,l,n,m$ be such that $g(x)\rel{X_0} T_i[n]$ and $g(y)\rel{X_0} T_l[m]$, we have $p_i(n)=p_l(m)\leftrightarrow x\rel{\cancel{L_j}} y$. Observe that if $\ran(g)$ is not almost contained in a single column, then $V$ is dense (i.e., for any finite sequence of finite injective partial maps $(p_i)_{i\leq m}$ there is a sequence $(q_i)_{i\leq n}$ with $n\geq m$ of injective partial maps so $p_i\subseteq q_i$ for $i\leq m$, and $(q_i)_{i\leq n}\in V$). But then by genericity of $(f_i)_{i\in \omega}$, it will meet $V$, which contradicts $g$ being a reduction of $L_j$ to $X$.
%	
%	
%Suppose otherwise and consider the following set of strings, which is c.e.\ in $W$,
%\[
%I:=\{\sigma\in\omega^{<\omega} : (\exists n,m)(\sigma(T_{i}[n])=T_j[m])) \}.
%\]
%As every $\rho\subseteq f$ extends into $I$, it must be the case that there is $\tau \in I$ such that $\tau\subseteq f$. This means that $f$ will eventually collapse a class from $i$th column with a class from the $j$th colum, so that $B$ cannot intersect infinitely many times both $T_i$ and $T_j$. 
\end{proof}

Let $i$ be fixed so that $\ran(g)\subseteq^* \col_i$. Since $\ran(g)$ intersects only finitely many columns, we can assume that it intersects the minimal possible number of columns. If $\ran(g)\subseteq \col_i$, then $L_j\leqc L_i$, which is a contradiction to $L_j$ and $L_i$ being inequivalent light minimal \ers. So, suppose that $\ran(g)$ intersects $\col_k$ for $k\neq i$. Let us consider the finite \er $Y=L_j\restriction g^{-1}(\col_k)$. If all $Y$-classes were computable, then we could adjust $g$ to send each of these sets to a representative of the same class in $\col_i$ contradicting that $g$ uses the minimal possible number of columns. So $Y\in \F\smallsetminus\I$ and $Y\leq L_j$ and $Y\leq L_k$. But Theorem~\ref{Id joins with finites} shows that there is a least upper bound $Z$ of $\Id$ and $Y$. Then $\Id<_c Z\leq L_j,L_k$ contradicting that $L_j$ and $L_k$ are inequivalent light minimal \ers.
%
%
%
%So, let $g$ be almost entirely contained in the $i$th column and let $k$ be another column hit by $g$. Note that $\range(g)\cap\col_k$ must contain at least 2 classes, as otherwise we could replace the use of the $k$th column by sending all $x$ so that $g(x)\cap \col_k$ to some fixed element in  $\col_i$ instead. Similarly, $\range(g)\cap\col_k$ cannot be equivalent to any $\Id_n$, as then we could again replace $g$ by a reduction using one less column. Let $A$ be $g^{-1}(\col_k)$. Therefore, $A$ is a member of $\F\smallsetminus \I$. But then $A\oplus \Id$ is the least upper bound of $A$ and $\Id$, by Lemma . But this contradicts that $L_j$ and $L_k$ are inequivalent and light minimal, since $A\leqc L_i, L_k$.
	\end{proof}

	If $\mathbf{a}$ is light, then let $M_{\mathbf{a}}$ be the set of light minimal degrees below $\mathbf{a}$.
	It follows that for every second order countable graph $G$, there are parameters $\mathbf{e},\mathbf{b}$ so that $(G,P(G))\cong (H_{\mathbf{e}}\cap M_{\mathbf{b}}, \mathcal{A})$ where $\mathcal{A}$ is the collection of sets $H_{\mathbf{e}}\cap M_{\mathbf{b}}\cap M_{\mathbf{a}}$ for various light degrees $\mathbf{a}$.

%	It follows that every countable graph $G$ can be encoded into $\Light$ as follows:
%\begin{enumerate}
%\item first, take a degree $\mathbf{c}$ which bounds no member of $\mathcal{F}\smallsetminus \mathcal{I}$ and such that $G_\mathbf{c}\cong G$ (Lemma \ref{c bounds not finite} guarantees that such $\mathbf{c}$ exists);
%\item then, use Lemma \ref{coding sets in light} to specify the domain of a substructure $X$ of $H_{\iota(c)}$ such that $G_{\mathbf{c}}\cong X$ (Lemma 	\ref{$G_c$ embeds into light} guaratees that such $X$ exists).
%\end{enumerate}	

%Finally, observe that Lemma \ref{coding sets in light} allows quantification over arbitrary subsets of vertices. Hence, any question about countable graphs can be effectively translated into $\Th(\Light)$. 

As remarked in Section \ref{coding graphs}, this suffices to conclude that the theory of $\bLight$ is computably isomorphic to second-order arithmetic.
\end{proof}

\bibliographystyle{plain}

\bibliography{references}

\end{document}